\documentclass[11pt]{elsarticle}

\usepackage{amsmath,amssymb,amscd,amsthm,esint}
\usepackage{graphics,mathrsfs,amsfonts}
\usepackage{mathtools}

\usepackage[margin=1in]{geometry}
\usepackage[colorlinks, linkcolor=blue, anchorcolor=blue, citecolor=blue]{hyperref}

\allowdisplaybreaks %allow math displays to break into two pages

\newtheorem{theorem}{Theorem}[section]
\newtheorem{corollary}{Corollary}[section]

\newtheorem{lemma}{Lemma}[section]
\newtheorem{remark}{Remark}[section]
\numberwithin{equation}{section}

\usepackage{amsopn}

\DeclareMathOperator*{\esssup}{ess\,sup}

\usepackage{accents}
\newcommand{\ubar}[1]{\underaccent{\bar}{#1}}
\usepackage{bm}

\begin{document}

\begin{frontmatter}

\title{Convergence rates in homogenization of parabolic systems with locally periodic coefficients}

\author{Yao Xu\fnref{myfootnote}}
\address{School of Mathematical Sciences, University of Chinese Academy of Sciences, Beijing, 100190, CHINA}
\ead{xuyao89@gmail.com, xuyao@ucas.ac.cn}
\fntext[myfootnote]{Supported partially by NNSF of China (No. 11971031).}

\begin{abstract}
  In this paper we study the quantitative homogenization of second-order parabolic systems with locally periodic (in both space and time) coefficients. The $O(\varepsilon)$ scale-invariant error estimate in $L^2(0, T; L^{\frac{2d}{d-1}}(\Omega))$ is established in $C^{1, 1}$ cylinders under minimum smoothness conditions on the coefficients. This process relies on critical estimates of smoothing operators. We also develop a new construction of flux correctors in the parabolic manner and a sharp estimate for temporal boundary layers.
\end{abstract}

\begin{keyword}
homogenization\sep parabolic systems\sep scale-invariant convergence rates\sep locally periodic coefficients
\MSC[2010] 35B27
\end{keyword}

\end{frontmatter}

%%%%%%%%%%%%%%%%%%%%%%%%%%%%%%%%%%%%%%%%%%%%%%%%%%%%%%%%%%%%%%%%%%%%%%%%%%%%%%%%%%%%%%%%%%%%%%%%%%%

\section{Introduction}\label{parabolic_sec_intro}

Let $\Omega$ be a bounded Lipschitz domain in $\mathbb{R}^d$, $d\geq 2$, and $T>0$. We consider the sharp convergence rate in the homogenization of the initial-boundary value problem with locally periodic coefficients
\begin{equation}\label{parabolic_intro_eq1}
\begin{cases}
  \partial_t u_\varepsilon+\mathcal{L}_\varepsilon u_\varepsilon =f  &\mathrm{in}~\Omega\times (0, T),\\
  u_\varepsilon=g &\mathrm{on}~\partial\Omega\times(0, T),\\
 u_\varepsilon=h  & \mathrm{on}~ \Omega\times\{t=0\},
\end{cases}
\end{equation}
where \begin{equation}\label{parabolic_intro_exp_L} \mathcal{L}_\varepsilon:= - \frac{\partial }{\partial x_i}\Big[A_{ij}^{\alpha\beta}\Big(x, t; \frac{x}{\varepsilon}, \frac{t}{\varepsilon^2}\Big) \frac{\partial }{\partial x_j}\Big],\quad \varepsilon>0,~1\leq i, j\leq d,~1\leq \alpha, \beta\leq m. \end{equation}
Note that the summation convention for repeated indices is used here and throughout the paper, and we may also omit the superscripts $\alpha, \beta$ if it is clear to understand. The coefficient matrix $A(x, t; y, \tau)=(A^{\alpha\beta}_{ij}(x, t; y, \tau))$ defined on $\Omega_T\times \mathbb{R}^{d+1}$ is assumed to be $1$-periodic in $(y, \tau)$, i.e.,
\begin{align}\label{parabolic_intro_cond_periodic}
  A(x, t; y+z, \tau+s)=A(x, t; y, \tau)\quad \mathrm{for~any}~(z, s)\in \mathbb{Z}^{d+1},
\end{align}
and satisfy the boundedness and ellipticity conditions
\begin{gather}\label{parabolic_intro_cond_elliptic}
  \begin{split}
\|A\|_{L^\infty(\Omega_T\times\mathbb{R}^{d+1})}&\leq 1/\mu,\\
A^{\alpha\beta}_{ij}(x, t; y, \tau) \xi_i^\alpha \xi_j^\beta &\geq \mu |\xi|^2
\end{split}
\end{gather}
for any $\xi\in \mathbb{R}^{d\times m}$ and a.e. $(x, t; y, \tau)\in \Omega_T\times \mathbb{R}^{d+1}$, where $\mu>0$ and $\Omega_T:=\Omega\times (0, T)$. We also assume that
\begin{equation}\label{parabolic_intro_cond_AC}
  A\in \mathscr{H}(\Omega_T; \bm{L^\infty}),
\end{equation}
where $\mathscr{H}(\Omega_T; \bm{L^\infty})$ is defined in Section \ref{parabolic_sec_four}.

System \eqref{parabolic_intro_eq1} is a simplified model describing  physical processes or chemical reactions taking place in composite materials, such as thermal conduction, the sulfate corrosion of concrete (see \cite{Bensoussan1978,Fatima2011,Muntean2013} for more complicated models). It applies to processes in heterogeneous media, while the strictly periodic model
\begin{equation}\label{parabolic_intro_eq_normal}
\partial_t u_\varepsilon-\mathrm{div}(A(x/\varepsilon, t/\varepsilon^2)\nabla u_\varepsilon)=f
\end{equation}
is only suitable for homogeneous phenomena. From a micro perspective, the microscopic pattern of system \eqref{parabolic_intro_eq1} is allowed to differ at different times and positions. Since real media in biomechanics and engineering are almost never homogeneous, \eqref{parabolic_intro_eq1} covers better what happens in practical applications.

System \eqref{parabolic_intro_eq1} contains two levels of scales, the macroscopic scales $(x, t)$ and the microscopic scales $(x/\varepsilon, t/\varepsilon^2)$. Usually, variables $(x, t)$ in $A(x, t; y, \tau)$ are called macroscopic variables, denoting space-time positions, while $(y, \tau)$ are microscopic variables representing fast variations at the microscopic structure. In both \eqref{parabolic_intro_eq1} and \eqref{parabolic_intro_eq_normal}, the scales of fast variations in space and time match naturally, that is, they are consistent with the intrinsic scaling $(x, t)\rightarrow(\lambda x, \lambda^2t)$ of second order parabolic equations. Generally, one may consider models where the microscopic scales are $(x/\varepsilon, t/\varepsilon^k)$ with $0<k<\infty$. In the periodic setting \eqref{parabolic_intro_eq_normal}, when $k=2$ we say the scales are self-similar, and when $k\neq 2$ they are non-self-similar \cite{Geng2020_nonself-similar}. Only in the case where $k=2$, the spatial scale and the temporal scale are homogenized simultaneously. More generally, problems with multiple (matching or mismatching) scales in space and time have also been introduced in \cite{Bensoussan1978, Holmbom2005,Pastukhova2009_parabolic_locally,Floden2014,Danielsson2019} and their references, where the qualitative homogenization of different types of matches was discussed widely. However, to the author's knowledge, quite few quantitative results have been known for parabolic equations with multiple scales. For recent results on the quantitative homogenization of elliptic systems with multiple scales, we refer to \cite{Xu2019_stratified2,Niu2020_Reiterated} and references therein.

Although the homogenized equation for \eqref{parabolic_intro_eq1} has been derived early in \cite{Bensoussan1978}, there is not much progress in the quantitative homogenization. The only notable literature is \cite{Pastukhova2009_parabolic_locally,Pastukhova2010_parabolic_locally}, where, for equation \eqref{parabolic_intro_eq1} with time-independent coefficients, the authors established the $O(\varepsilon^{1/2})$ estimate of the operator exponential $e^{-t\mathcal{L}_\varepsilon}$ to its limit in $L^2$ for each $t\geq 1$. More recently, the rate in $L^2(\Omega_T)$, as well as the full-scale Lipschitz estimates, for systems with non-self-similar scales is discussed widely under strong smoothness assumptions on the coefficients in \cite{Geng2021_locally}. We also refer to the series of work \cite{Bench2004, Bench2005, Bench2007} for the qualitative pointwise convergence results of the same equation obtained by a probabilistic approach.

Under assumptions \eqref{parabolic_intro_cond_elliptic}--\eqref{parabolic_intro_cond_AC}, the coefficient $A(x, t; x/\varepsilon, t/\varepsilon^2)$ of system \eqref{parabolic_intro_eq1} is measurable on $\Omega_T$. We also assume that $f, g, h$ satisfy suitable conditions so that problem \eqref{parabolic_intro_eq1} admits a unique weak solution $u_\varepsilon$. As is well known, $u_\varepsilon$ converges to a function $u_0$ weakly in $L^2(0, T; H^1(\Omega))$ as $\varepsilon\rightarrow 0$, where $u_0$ is the solution of the homogenized problem given by
\begin{equation}\label{parabolic_intro_eq_u0}
\begin{cases}
  \partial_t u_0+\mathcal{L}_0 u_0 =f  &\mathrm{in}~\Omega\times (0, T),\\
  u_0=g &\mathrm{on}~\partial\Omega\times(0, T),\\
 u_0=h  & \mathrm{on}~ \Omega\times\{t=0\},
\end{cases}
\end{equation}
and $\mathcal{L}_0$ is a divergence type elliptic operator with variable coefficients determined by solving unit cell problems at each point of the domain (see Section \ref{parabolic_sec_corrector}).

Our main goal is to establish the optimal convergence rate of $u_\varepsilon$ to $u_0$. Convergence rate is a core subject of homogenization and has aroused much interest in the past few years. So far, various results about convergence rates have been gained for parabolic systems with time-independent or time-dependent periodic coefficients. The reader may consult \cite{Zhikov2006_parabolic, Meshkova2016_parabolic, Geng2017_Convergence,Xu2017_parabolic_elasticity, Niu2019_refined, Geng2021_locally} and their references. However, almost all these rates are in the sense of $L^2$. In this paper, we establish a scale-invariant result for system \eqref{parabolic_intro_eq1} in $L^2(0, T; L^{p_0}(\Omega))$ with $p_0=\frac{2d}{d-1}$ under quite general conditions.

\begin{theorem}\label{parabolic_conver_thm_conver}
  Let $\Omega$ be a bounded $C^{1, 1}$ domain in $\mathbb{R}^d$, $d\geq 2$ and $T>0$. Assume that $A$ satisfies  \eqref{parabolic_intro_cond_elliptic}--\eqref{parabolic_intro_cond_AC}. Let $u_\varepsilon$ and $u_0$ be the weak solutions to problems \eqref{parabolic_intro_eq1} and \eqref{parabolic_intro_eq_u0}, respectively. Suppose further $u_0\in L^2(0, T; W^{2, q_0}(\Omega))$, $\partial_t u_0\in L^2(0, T; L^{q_0}(\Omega))$ with $q_0:=\frac{2d}{d+1}$. Then
  \begin{align}\label{parabolic_conver_Lp_conver_delta}
    \|u_\varepsilon-u_0\|_{L^2(0, T; L^{p_0}(\Omega))}\leq C\varepsilon\{\|\nabla u_0\|_{L^2(0, T; \dot{W}^{1, q_0}(\Omega))}+\|\partial_t u_0\|_{L^2(0, T; L^{q_0}(\Omega))}\},
  \end{align}
  where $p_0:=\frac{2d}{d-1}$, $C$ depends only on $d, m, n, \mu, \Omega, A$ and
  \begin{align*}
    \|u\|_{\dot{W}^{1, q_0}(\Omega)}=\|u\|_{L^{\frac{d q_0}{d-q_0}}(\Omega)}+\|\nabla u\|_{L^{q_0}(\Omega)}.
  \end{align*}
\end{theorem}

Theorem \ref{parabolic_conver_thm_conver} extends the result of \cite{Niu2019_refined}, a similar error estimate for \eqref{parabolic_intro_eq_normal}, to the locally periodic setting. It is remarkable that estimate \eqref{parabolic_conver_Lp_conver_delta} is scale-invariant under the parabolic rescaling $u(x, t)\rightarrow \lambda^{-2}u(\lambda x, \lambda^2t)$ and the constant $C$ in \eqref{parabolic_conver_Lp_conver_delta} is independent of the size of $\Omega$. This estimate is more elegant than that of \cite{Niu2019_refined} given as (with $g=0$) $$\|u_\varepsilon-u_0\|_{L^2(0, T; L^{p_0}(\Omega))}\leq C\varepsilon\{\|u_0\|_{L^2(0, T; W^{2, q_0}(\Omega))}+\|f\|_{L^2(0, T; L^{q_0}(\Omega))}+\|h\|_{H^1(\Omega)}\},$$where the scale of $\|h\|_{H^1(\Omega)}$ does not coincide with the other terms when doing scaling. The earliest work on these kinds of scale-invariant results in homogenization should be attributed to Z. Shen who established the rate for elliptic systems with periodic coefficients in the noted book \cite{Shen2018_book}. Later in \cite{Xu2019_stratified,Xu2019_stratified2}, the scale-invariant error estimates were extended to elliptic systems with stratified coefficients $A(x, \rho(x)/\varepsilon)$ under rather general smoothness assumptions.

In one sense, the smoothness condition \eqref{parabolic_intro_cond_AC} means $A(x, t; y, \tau)$ is $1$-order differentiable in $x$ and $\frac{1}{2}$-order differentiable in $t$, which coincides with the regularity of general parabolic equations. Also the space $\mathscr{H}(\Omega_T; \bm{L^\infty})$ has the same scale as $L^\infty(\Omega_T\times \mathbb{T}^{d+1})$. From the viewpoint of calculations, this smoothness condition is minimum to guarantee the $O(\varepsilon)$ error estimate. 

Before describing the strategies and skills used in the paper, we introduce the notation
\begin{align}
  \label{notation_sup_ve}
  \phi^\varepsilon(x, t):=\phi(x, t; x/\varepsilon, t/\varepsilon^2),
\end{align}
which gives
\begin{align}\label{parabolic_intro_iden_composition}
  \begin{split}
    \partial_{x_i}(\phi^\varepsilon(x, t))&=(\partial_{x_i}\phi)^\varepsilon(x, t)+\varepsilon^{-1}(\partial_{y_i}\phi)^\varepsilon(x, t),\\
\partial_t(\phi^\varepsilon(x, t))&=(\partial_t\phi)^{\varepsilon}(x, t)+\varepsilon^{-2}(\partial_\tau\phi)^\varepsilon(x, t).
  \end{split}
\end{align}

The main difficulties of the paper are essentially caused by the feature of two scales. As seen in the formal asymptotic expansion
\begin{align*}
  u_\varepsilon(x, t)=u_0(x, t)+\varepsilon\chi_j(x, t; x/\varepsilon, t/\varepsilon^2)\partial_{j}u_0(x, t)+\cdots,
\end{align*}
the first-order term $\varepsilon\chi_j(x, t; x/\varepsilon, t/\varepsilon^2)\partial_{j}u_0(x, t)$ may not even be measurable on $\Omega_T$, as $\chi$ is not regular enough. To handle this problem, as in \cite{Xu2019_stratified2} for elliptic systems, we introduce the smoothing operator $S_\varepsilon$ w.r.t. macroscopic variables $(x, t)$. It makes $S_\varepsilon(g)$ smooth in $(x, t)$ for any $g(x, t; y, \tau)$ which is $1$-periodic in $(y, \tau)$, thereby ensuring the measurability of $[S_\varepsilon(g)]^\varepsilon(x, t)=S_\varepsilon(g)(x, t; x/\varepsilon, t/\varepsilon^2)$. Moreover, by Fubini's theorem this operator also helps us separate $(y, \tau)$ from $(x, t)$ in the coupled form $(x, t; x/\varepsilon, t/\varepsilon^2)$. On the other hand, since $S_\varepsilon$ is introduced to $g$ auxiliarily, it is necessary to control the difference $g-S_\varepsilon(g)$ (mostly in the case $g=A$). The idea is to write this difference into convolutions which act just like smoothing operators. In fact, it involves the differences in both space and time, namely,
\begin{align*}
  g-S_\varepsilon(g)=g-S^t_\varepsilon(g)+S^t_\varepsilon(g)-S_\varepsilon(g),
\end{align*}
where $S^t_\varepsilon$ is the smoothing operator w.r.t. $t$ only (see \eqref{parabolic_pre_def_Si}). The latter term is the difference in space and, by Poincar\'{e}'s inequality, it could be dominated by the ``convolution''
\begin{align*}
  \int_{B(x, \varepsilon)}|S^t_\varepsilon(\nabla_xg)(\omega, t; y, \tau)||\omega-x|^{1-d}d\omega.
\end{align*}
The first term can be written formally into
\begin{align*}
  g(x, t; y, \tau)-S^t_\varepsilon(g)(x, t; y, \tau)=-\varepsilon^2\int_0^1\int_{\mathbb{R}}g(x, \varsigma; y, \tau)\cdot\partial_t\varPhi_{\varepsilon^2\theta}(t-\varsigma)d\varsigma d\theta,
\end{align*}
where $\Phi_{\varepsilon^2\theta}$ is a proper kernel. Both of these two terms can be controlled well by the critical estimates established in Section \ref{parabolic_subsec_smoothing}.

It should be pointed out that, due to the feature of the coefficient, the rate in $L^2(0, T; L^{p_0}(\Omega))$ involves many subtle inhomogeneous $L^q_tL^p_xL^s_{\tau}L^r_y$-type estimates of four-variable functions. This urges us to perform each step accurately in the right format, and by this reason, the process is much more delicate than that of elliptic problems in \cite{Xu2019_stratified2}.

On the other hand, we introduce a new construction of flux correctors. In \cite{Geng2017_Convergence}, flux correctors were constructed in an elliptic manner by lifting the function $B_{ij}$ in both $y$ and $\tau$, which results in high regularity in $\tau$ but low regularity in $y$ (especially for $\mathfrak{B}_{(d+1)ij}$). It does not work for higher-order parabolic systems, as more regularity in $y$ is required. Later, in \cite{Niu2018_parabolic} when dealing with higher-order systems, flux correctors were constructed by lifting the regularity w.r.t. space only. This provides enough regularity in $y$, but no regularity in $\tau$. Neither of them is applicable to our setting, as the microscopic regularities involved are very subtle. To this end, we construct flux correctors in a parabolic manner, in which way the regularities of flux correctors are the same as and even better than correctors. This construction seems more natural and is also valid for higher-order systems.

Another novelty of the paper is a new estimate of temporal boundary layers. Compared to the method in \cite{Niu2019_refined}, it provides better estimates but requires no restriction on $g$ and $h$. The estimate is based on a sharp embedding result for $u_0$ and it improves the estimate used in \cite{Geng2020_nonself-similar}, where one half of the power of temporal layers was in fact lost.

We now describe the outline of the paper. Section \ref{parabolic_sec_pre} contains several parts of preliminaries. In Section \ref{parabolic_sec_four}, various vector-valued spaces of multi-variable functions are introduced, which are suitable tools to describe the properties of correctors and flux correctors. Section \ref{sec_homog-sobol} provides some homogeneous Sobolev spaces. Next in Sections \ref{parabolic_sec_corrector} and \ref{parabolic_sec_flux-correctors}, correctors $\chi$ and flux correctors $\mathfrak{B}$, together with their regularities, are studied. Section \ref{parabolic_subsec_smoothing} is devoted to a series of critical estimates for the smoothing operator $S_\varepsilon$. Afterwards, a sharp embedding result for $u_0$ is built in Section \ref{parabolic_sec_boundary-layers}, along with a corollary on the estimate of boundary layers.

These results are applied in Section \ref{parabolic_sec_conv-rates} to establish the $O(\varepsilon^{1/2})$ rate in $L^2(0, T; H^1(\Omega))$ and the $O(\varepsilon)$ rate in $L^2(0, T; L^{p_0}(\Omega))$. More precisely, the $O(\varepsilon^{1/2})$ estimate is established for the auxiliary function
\begin{align}\label{parabolic_intro_def_w_ve}
  \begin{split}
  w_\varepsilon&=u_\varepsilon-u_0-\varepsilon S_\varepsilon(S_\varepsilon(\chi) S_\varepsilon(\nabla u_0)\eta_\varepsilon)(x, t; \frac{x}{\varepsilon}, \frac{t}{\varepsilon^2})\\&\quad-\varepsilon^2\partial_{x_k}S_\varepsilon(S_\varepsilon(\mathfrak{B}_{(d+1)kj})S_\varepsilon(\partial_j u_0)\eta_\varepsilon)(x, t; \frac{x}{\varepsilon}, \frac{t}{\varepsilon^2}),
  \end{split}
\end{align}
where $\eta_\varepsilon$ is a cut-off function. Note that in this process those three $S_\varepsilon$ in the last two terms of \eqref{parabolic_intro_def_w_ve} play different roles: the first $S_\varepsilon$ is mainly used to maintain the measurability on $\Omega_T$ and control rapidly oscillating factors via Fubini's theorem; the second operator helps us to reduce the smoothness assumption on $A$ in $t$; and the third one reduces the regularities of $u_0$ at the cost of the power of $\varepsilon$.

To derive the rate of $u_\varepsilon$ to $u_0$ stated in Theorem \ref{parabolic_conver_thm_conver}, we adopt the classical duality argument, where the solution of the dual problem satisfies $L^q$-$L^p$ estimates in $\Omega_T$. The main challenge in this process lies in the term
\begin{align*}
\Big|\iint_{\Omega_T}[A-\widehat{A}]^\varepsilon\cdot(\nabla u_0-S_\varepsilon(\nabla u_0))\cdot[S_\varepsilon(\nabla_y\widetilde{\chi}^* K_\varepsilon(\nabla v_0))]^\varepsilon\Big|,
\end{align*}
where $\chi^*$ is the corrector of the dual problem. For this term, the technique used in \cite{Niu2019_refined} is no longer in force, since $\chi^*$ is now a four-variable function and it does not have enough regularity to rearrange arbitrarily the order of integrals in mixed norms. Instead, the idea is, formally speaking, to transfer the gradient in $\nabla_y\widetilde{\chi}^*$ to $u_0$ by integrating by parts, which is carried out by a regularity lifting argument (see Lemma \ref{parabolic_conver_lem_u0-Su0}).

Throughout this paper, unless otherwise stated, we will use $C$ to denote any positive constant which may depend on $d, m, \mu$. It should be understood that $C$ may differ from each other even in the same line. We also use the notation $\fint_E f:=(1/|E|)\int_E f$ for the integral average of $f$ over $E$.

\section{Preliminaries and correctors}\label{parabolic_sec_pre}
In this section, we introduce briefly some vector-valued function spaces with mixed norms for four-variable functions. Correctors and flux correctors are also introduced. 

\subsection{Multi-variable function spaces}
\label{parabolic_sec_four}

Recall that, a vector-valued function space is defined via the \emph{strong measurability}. Precisely, given a measure space $(E, \mu)$ and a Banach space $(B, \|\cdot\|_B)$, $L^p(E; B)$ is the space of \emph{strongly measurable} vector-valued functions from $E$ into $B$ satisfying $\|h\|_{L^p(E; B)}<\infty$, where
\begin{align*}%\label{parabolic_pre_notation_vector}
  \|h\|_{L^p(E; B)}:=
  \begin{cases}
    \Big(\int_E\|h\|_B^p~d\mu\Big)^{1/p}& \mathrm{if}~p<\infty,\\[0.2cm]
    \esssup\limits_E\{\|h\|_B\} &\mathrm{if}~p=\infty.
  \end{cases}
\end{align*}
Here the \emph{strong measurability} means that elements can be approximated almost everywhere by countably-valued functions. Following this framework, we introduce some vector-valued function spaces with mixed norms.

For $k\in \mathbb{N^+}$, we denote the $k$-dimensional torus of length $1$ by $\mathbb{T}^{k}$. Then functions of period $1$ on $\mathbb{R}^k$ may be regarded as functions on $\mathbb{T}^k$. Detailedly, $L^r(\mathbb{T}^{d+1})$ and $W^{1, r}(\mathbb{T}^{d+1})$ ($1\leq r\leq \infty$) are the subspaces of $L^r_{\mathrm{loc}}(\mathbb{R}^{d+1})$ and $W^{1, r}_{\mathrm{loc}}(\mathbb{R}^{d+1})$ whose elements are $1$-periodic, respectively. Moreover, $L_{\mathfrak{m}}^s(\mathbb{T}^1; L^r(\mathbb{T}^d))$ $(1\leq r, s\leq \infty)$ denotes the anisotropic space of $1$-periodic \emph{measurable} functions on $\mathbb{T}^{d+1}$ endowed with the norm $\|\cdot\|_{L^s(\mathbb{T}^1; L^r(\mathbb{T}^d))}$ (see Section 2.1 in \cite{Xu2019_stratified2}). Note that the main distinction between $L_{\mathfrak{m}}^s(\mathbb{T}^1; L^r(\mathbb{T}^d))$ and $L^s(\mathbb{T}^1; L^r(\mathbb{T}^d))$, the space of \emph{strongly measurable} vector-valued functions from $\mathbb{T}^1$ into $L^r(\mathbb{T}^d)$, concentrates on the measurability of elements, and these two spaces are equivalent if $r<\infty$. Furthermore, if $s=r$, $L_{\mathfrak{m}}^s(\mathbb{T}^1; L^r(\mathbb{T}^d))=L^r(\mathbb{T}^{d+1})$. For the sake of brevity, we may write
\begin{align}\label{parabolic_pre_notation_range}
  \bm{L^r}:=L^r(\mathbb{T}^{d+1}),\quad \bm{W^{1, r}}:=W^{1, r}(\mathbb{T}^{d+1}),\quad \bm{L^{s, r}}:=L_{\mathfrak{m}}^s(\mathbb{T}^1; L^r(\mathbb{T}^d)),
\end{align}
as the domain is invariant and the periodicity is always required. All of these spaces are equipped with the product measurability on $\mathbb{T}^{d+1}$ and will be the ranges of vector-valued function spaces in our study.

Now set $E=\Omega\times I$, where $I$ is a finite closed interval and $\Omega$ is an open set of $\mathbb{R}^{d}$ for the moment. Given $B$ as any space in \eqref{parabolic_pre_notation_range}, we discuss about vector-valued function spaces from $E$ into $B$ with mixed norms. For $1\leq p, q\leq \infty$, we say $h(x, t; y, \tau)$ belongs to $L^{q, p}(E; B)$ if $h$ is \emph{strongly measurable} from $E$ into $B$ and satisfies $\|h\|_{L^{q, p}(E; B)}<\infty$, where $$\|h\|_{L^{q, p}(E; B)}:=\|h\|_{L^q(I; L^p(\Omega; B))}.$$ Note that $L^{q, p}(E; B)$ is a Banach space under norm $\|\cdot\|_{L^{q, p}(E; B)}$ and, if $q=p$, $L^{q, p}(E; B)=L^p(E; B)$. We also point out that the elements of $L^{q, p}(E; B)$ are measurable w.r.t. the Lebesgue (product) measure on $E\times \mathbb{R}^{d+1}$. Moreover, if $h(x, t; y, \tau)\in L^{1}(E; \bm{L^\infty})$, then $h(x, t; \frac{x}{\varepsilon}, \frac{t}{\varepsilon^2})$ is a measurable function of $(x, t)$ on $E$, where we have regarded $h$ in $h(x, t; \frac{x}{\varepsilon}, \frac{t}{\varepsilon^2})$ to be its precise representative (see \cite{Xu2019_stratified2} for more details)
\begin{align*}
  h^*(x, t; y, \tau)=
  \begin{cases}
  \lim\limits_{r\rightarrow 0}\fint_{\{|(x', t'; y', \tau')-(x, t; y, \tau)|\leq r\}}h(x', t'; y', \tau')dx' dt' dy' d\tau'&\textrm{if the limit exists,}\\0&\textrm{otherwise}.
\end{cases}
\end{align*}

To keep presentations simple, we may use the following notations for multiple integrals of $h(x, t; y, \tau)$ on $E\times \mathbb{T}^{d+1}$:
\begin{align}\label{parabolic_pre_notation_multiple}
  \begin{split}
      \|h\|_{L^r_y}(x, t; \tau)&:=\|h(x, t; \cdot, \tau)\|_{L^r(\mathbb{T}^d)},\quad\quad\|h\|_{L^{s, r}_{\tau, y}}(x, t):=\|h(x, t; \cdot, \cdot)\|_{\bm{L^{s, r}}},\\
  \|h\|_{L^{p, s, r}_{x, \tau, y}(\Omega)}(t)&:=\|h(\cdot, t; \cdot, \cdot)\|_{L^p(\Omega; \bm{L^{s,r}})},\quad \|h\|_{L^{q, p, s, r}_{t, x, \tau, y}(E)}:=\|h\|_{L^{q, p}(E; \bm{L^{s, r}})}.
  \end{split}
\end{align}
In particular, if $h(x ,t; y, \tau)$ is independent of $(y, \tau)$, we have $$\|h\|_{L^{q, p}(E)}=\|h\|_{L^q(I; L^p(\Omega))}.$$

On the other hand, due to Fubini's theorem, we can define the weak derivative w.r.t. variable $x$ as a distribution. Indeed, if $h(x, t; y, \tau)\in L^1_{\mathrm{loc}}(E\times \mathbb{R}^{d+1})$, for a.e. $(t; y, \tau)$, we define $\nabla_x h(\cdot, t; y, \tau)$ as a distribution on $\Omega$ by
\begin{align*}
  \langle\nabla_xh(x, t; y, \tau), \phi(x)\rangle_\Omega=-\int_\Omega h(x, t; y, \tau)\nabla_x\phi(x)~dx\quad \mathrm{for}~\phi\in C_0^\infty(\Omega).
\end{align*}
Furthermore, for $0<\sigma<1$ and $1\leq q< \infty$, the fractional Sobolev-Slobodeckij space $W^{\sigma, q}(I; X)$ is defined to be the set of vector-valued function $h$ from $I$ into Banach space $X$ satisfying $\|h\|_{W^{\sigma, q}(I; X)}:=\|h\|_{L^q(I; X)}+[h]_{W^{\sigma, q}(I; X)}<\infty$, where
\begin{align}
  [h]_{W^{\sigma, q}(I; X)}:=\Big(\int_I\int_I\frac{\|h(t_1)-h(t_2)\|_X^q}{|t_1-t_2|^{1+\sigma q}}dt_1 dt_2\Big)^{\frac{1}{q}}.\label{parabolic_pre_def_seminorm}
\end{align}
$W^{\sigma, q}(I; X)$ is a Banach space under the norm $\|\cdot\|_{W^{\sigma, q}(I; X)}$.

Lastly, we say $h\in \mathscr{H}(E; B)$ if
\begin{align}
  \label{parabolic_pre_def_C_seminorm}
[h]_{\mathscr{H}(E; B)}:=[h]_{W^{\frac{1}{2}, 2}(I; L^\infty(\Omega; B))}+\|\nabla_x h\|_{L^{\infty, d}(E; B)}<\infty.
\end{align}
Note that $[\,\cdot\,]_{\mathscr{H}(E; B)}$ is a semi-norm. In a sense, $h\in\mathscr{H}(E; B)$ means that $h$ has $\frac{1}{2}$-order derivative in time and $1$-order derivative in space. If $h\in \mathscr{H}(E; \bm{L^\infty})$, then $h(x, t; \frac{x}{\varepsilon}, \frac{t}{\varepsilon^2})$ and $\nabla_x h(x, t; \frac{x}{\varepsilon}, \frac{t}{\varepsilon^2})$ are measurable functions on $E$ taking precise representatives into consideration.

\subsection{Homogeneous Sobolev spaces}
\label{sec_homog-sobol}
To make the estimates scale-invariant, we introduce the following homogeneous Sobolev spaces. Denote henceforth
\begin{align*}
  p^*:=\frac{dp}{d-p}\quad\mathrm{for}~1\leq p<d,\qquad  p':=\frac{p}{p-1}\quad \mathrm{for}~1\leq p\leq \infty. 
\end{align*}
For a domain $\Omega\subset\mathbb{R}^d$ and $1\leq p<d$, set
\begin{align*}
  \dot{W}^{1, p}(\Omega):=\{u\in L^{p^*}(\Omega): \|\nabla u\|_{L^p(\Omega)}<\infty\}
\end{align*}
endowed with the norm $\|u\|_{\dot{W}^{1, p}(\Omega)}=\|u\|_{L^{p^*}(\Omega)}+\|\nabla u\|_{L^p(\Omega)}$. For $1\leq p<\infty$, set
\begin{align*}
  \dot{W}^{1, p}_0(\Omega):=\textrm{the completion of }C_c^\infty(\Omega)\textrm{ under }\|\nabla u\|_{L^p(\Omega)}.
\end{align*}
Note that $\|\nabla u\|_{L^p(\Omega)}$ is a norm in $\dot{W}^{1, p}_0(\Omega)$ and $\dot{W}^{1, p}(\mathbb{R}^d)=\dot{W}^{1, p}_0(\mathbb{R}^d)$ for $1\leq p<d$. For $1<p\leq \infty$, we denote by $\dot{W}^{-1, p}(\Omega)$ the dual space of $\dot{W}^{1, p'}_0(\Omega)$. It is not hard to show that $F\in \dot{W}^{-1, p}(\Omega)$ can be written into $F=\sum_{j=1}^d\partial_{x_j}f_j$ for some $f_j\in L^p(\Omega)$. 

\subsection{Correctors and effective coefficients}\label{parabolic_sec_corrector}

For $1\leq \beta\leq m, 1\leq j\leq d$, set $P_j^\beta=P_j^\beta(y):=y_je^\beta$, where $e^\beta=(0, \dots, 1, \dots, 0)$ with $1$ in the $\beta$-th position. According to the qualitative results of homogenization in \cite{Bensoussan1978}, the matrix of correctors $\chi^\beta_j(x, t; y, \tau)=(\chi^{\gamma\beta}_j(x, t; y, \tau))$ is given by the following system
\begin{equation}\label{parabolic_pre_eq_chi}
\begin{cases}
  \partial_\tau\chi_j^\beta(x, t; y, \tau)+\mathcal{L}^{x, t}\chi^\beta_j(x, t; y, \tau) =-\mathcal{L}^{x, t}(P_j^\beta)\quad \mathrm{in}~ \mathbb{R}^{d+1},\\
\chi^\beta_j~\textrm{is $1$-periodic in}~(y, \tau)~\textrm{and}~\int_{\mathbb{T}^{d+1}}\chi^\beta_j(x, t; y, \tau)~dyd\tau=0~\textrm{for a.e. }(x, t)\in \Omega_T,
\end{cases}
\end{equation}
where
\begin{align*}
  \mathcal{L}^{x, t}:= -\frac{\partial }{\partial y_k}\Big[A_{kl}^{\alpha\gamma}(x, t; y, \tau) \frac{\partial }{\partial y_l}\cdot\Big],
\end{align*}
and $x, t$ play the role of parameters. By Fubini's theorem, equation \eqref{parabolic_pre_eq_chi} is well-posed for a.e. $(x, t)$ under assumptions \eqref{parabolic_intro_cond_periodic}--\eqref{parabolic_intro_cond_elliptic}.

On the other hand, we can introduce the matrix of correctors $\chi_j^{*\beta}$ for the operator $-\partial_t+\mathcal{L}_\varepsilon^*$, where $\mathcal{L}^*_\varepsilon$ is the adjoint operator of $\mathcal{L}_\varepsilon$, given by \eqref{parabolic_intro_exp_L} with $A$ replaced by its adjoint $A^*$. Then $\chi_j^{*\beta}$ solves the cell problem
\begin{align}\label{parabolic_pre_eq_chi_dual}
-\partial_\tau\chi_j^{*\beta}(x, t; y, \tau)+(\mathcal{L}^{x, t})^*\chi^{*\beta}_j(x, t; y, \tau)=-(\mathcal{L}^{x, t})^*P_j^\beta \quad  \mathrm{in}~\mathbb{T}^{d+1}.
\end{align}

The matrix of effective coefficients for $\partial_t+\mathcal{L}_\varepsilon$ is defined by \begin{align}\label{parabolic_pre_def_hatA}\widehat{A}^{\alpha\beta}_{ij}(x, t):=\iint_{\mathbb{T}^{d+1}}A^{\alpha\beta}_{ij}(x, t; y, \tau)+A^{\alpha\gamma}_{ik}(x, t; y, \tau)\partial_{y_k} \chi^{\gamma\beta}_j(x, t; y, \tau)~dyd\tau.\end{align}
One can verify that $\widehat{A}$ satisfies the ellipticity condition and $(\widehat{A})^*=\widehat{A^*}$, where $\widehat{A^*}$ is the matrix of effective coefficients for $-\partial_t+\mathcal{L}_\varepsilon^*$.

Inherited from $A$, we have the following estimates on $\chi$.

\begin{lemma}\label{parabolic_pre_lem_chi}
  Suppose that $A$ satisfies \eqref{parabolic_intro_cond_periodic}--\,\eqref{parabolic_intro_cond_elliptic}. Then
  \begin{itemize}
  \item [1).] there exists $\bar{q}>2$, depending only on $\mu$, such that, for a.e. $(x, t)\in \Omega_T$, $\chi(x, t; \cdot, \cdot)\in \mathbb{B}_1:=L^{\bar{q}}(\mathbb{T}^1; W^{1, \bar{q}}(\mathbb{T}^d))\cap L^\infty(\mathbb{T}^1; L^{\bar{q}}(\mathbb{T}^d))$, and it holds that
  \begin{align*}
    \|\chi(x, t)\|_{\mathbb{B}_1}\leq C,\quad \|\nabla_x\chi(x, t)\|_{\mathbb{B}_1}\leq C\|\nabla_xA(x, t)\|_{L^\infty(\mathbb{T}^{d+1})},
  \end{align*}
  and for a.e. $x_1, x_2\in \Omega$, $t_1, t_2\in [0, T]$,
  \begin{align*}
    \begin{split}
       \|\chi(x_1, t)-\chi(x_2, t)\|_{\mathbb{B}_1}\leq C\|A(x_1, t)-A(x_2, t)\|_{L^\infty(\mathbb{T}^{d+1})},\\ \|\chi(x, t_1)-\chi(x, t_2)\|_{\mathbb{B}_1}\leq C\|A(x, t_1)-A(x, t_2)\|_{L^\infty(\mathbb{T}^{d+1})},
    \end{split}
  \end{align*}
  where $C$ depends only on $\mu$;
\item [2).] furthermore, if in addition $A$ satisfies \eqref{parabolic_intro_cond_AC}, we have $\chi\in L^\infty(\Omega_T; \mathbb{B}_1)\cap \mathscr{H}(\Omega_T; \mathbb{B}_1)$ and
  \begin{align*}\|\chi\|_{L^\infty(\Omega_T; \mathbb{B}_1)}\leq C,\quad [\chi]_{\mathscr{H}(\Omega_T; \mathbb{B}_1)}\leq C[A]_{\mathscr{H}(\Omega_T; \bm{L^\infty})}.\end{align*}
    \end{itemize}
\end{lemma}

\begin{proof}
The first part follows from Meyers-type estimate for parabolic systems in \cite{Campanato1980_parabolic,Armstrong2018_parabolic} together with the equation of $\chi$. The second part mainly asserts the \emph{strong measurability} of $\chi$ and $\nabla_x\chi$ from $\Omega_T$ into $\mathbb{B}_1$, which can be proved by approximating as Lemma 2.4 and Corollary 2.2 in \cite{Xu2019_stratified2}. We omit the details.
\end{proof}
\begin{corollary}\label{parabolic_pre_coro_hatA}
  Under assumptions \eqref{parabolic_intro_cond_elliptic}--\,\eqref{parabolic_intro_cond_AC}, we have $$\widehat{A}^{\alpha\beta}_{ij}(x, t)\in W^{\frac{1}{2}, 2}(0, T; L^\infty(\Omega))\cap L^\infty(0, T; W^{1, d}(\Omega)).$$
\end{corollary}

\subsection{Flux correctors}
\label{parabolic_sec_flux-correctors}

  In this subsection, we introduce flux correctors for $\partial_t+\mathcal{L}_\varepsilon$ in a new manner of parabolic type. Denote the indices ranging between $1$ and $d+1$ by underlined symbols, such as $\ubar{i}$. In other words, $\ubar{i}$ may equal $1, \dots, d+1$. As like in \cite{Geng2017_Convergence, Niu2019_refined}, for $1\leq j\leq d$, set
  \begin{align}\label{parabolic_def_B}B^{\alpha\beta}_{\ubar{i}j}(x, t; y, \tau):=\begin{cases}A^{\alpha\beta}_{\ubar{i}j}(x, t; y, \tau)+A^{\alpha\gamma}_{\ubar{i}k} \partial_{y_k}\chi_j^{\gamma\beta}(x, t; y, \tau)-\widehat{A}^{\alpha\beta}_{\ubar{i}j}(x, t), & \mathrm{if}~1\leq \ubar{i}\leq d,\\-\chi_{j}^{\alpha\beta}(x, t; y, \tau), &\mathrm{if}~\ubar{i}=d+1.\end{cases}\end{align}
  Then \begin{gather*}\|B_{\ubar{i}j}\|_{L^\infty(\Omega_T; \bm{L^{\bar{q}}})}\leq C,\quad [B_{\ubar{i}j}]_{\mathscr{H}(\Omega_T; \bm{L^{\bar{q}}})}\leq C[A]_{\mathscr{H}(\Omega_T; \bm{L^\infty})} \quad \textrm{for }1\leq \ubar{i}\leq d+1,\\\textrm{and further } \|\nabla_yB_{(d+1)j}\|_{L^\infty(\Omega_T; \bm{L^{\bar{q}}})}\leq C,\quad [\nabla_yB_{(d+1)j}]_{\mathscr{H}(\Omega_T; \bm{L^{\bar{q}}})}\leq C[A]_{\mathscr{H}(\Omega_T; \bm{L^\infty})}.\end{gather*}
\begin{lemma}\label{parabolic_pre_lem_fkB}
  There exist $\mathfrak{B}^{\alpha\beta}_{\ubar{k}\ubar{i}j}(x, t; y, \tau)$, $1\leq \ubar{i}, \ubar{k}\leq d+1$, $1\leq j\leq d$, $1\leq \alpha, \beta\leq m$, which are $1$-periodic in $(y, \tau)$, such that
  \begin{align*}
    %\label{parabolic_pre_iden_fkB}
    \partial_{y_k} \mathfrak{B}^{\alpha\beta}_{k\ubar{i}j}(x, t; y, \tau)+\partial_\tau \mathfrak{B}^{\alpha\beta}_{(d+1)\ubar{i}j}(x, t; y, \tau)=B^{\alpha\beta}_{\ubar{i}j}(x, t; y, \tau) \quad\textrm{and}\quad \mathfrak{B}^{\alpha\beta}_{\ubar{k}\ubar{i}j}=- \mathfrak{B}^{\alpha\beta}_{\ubar{i}\ubar{k}j}.
  \end{align*}
  Furthermore, there exists a constant $C$, depending only on $\mu$, such that, for $1\leq j\leq d$,  \begin{align*}\|\mathfrak{B}_{\ubar{k}\ubar{i}j}\|_{L^\infty(\Omega_T; \mathbb{B}_1)}\leq C,\quad [\mathfrak{B}_{\ubar{k}\ubar{i}j}]_{\mathscr{H}(\Omega_T; \mathbb{B}_1)}\leq C[A]_{\mathscr{H}(\Omega_T; \bm{L^\infty})},\quad& \textrm{if } 1\leq \ubar{k}, \ubar{i}\leq d,\\    \|\mathfrak{B}_{\ubar{k}\ubar{i}j}\|_{L^\infty(\Omega_T; \mathbb{B}_2)}\leq C,\quad [\mathfrak{B}_{\ubar{k}\ubar{i}j}]_{\mathscr{H}(\Omega_T; \mathbb{B}_2)}\leq C[A]_{\mathscr{H}(\Omega_T; \bm{L^\infty})},\quad& \textrm{if } \ubar{k} \textrm{ or } \ubar{i} = d+1,
  \end{align*}
  where $\mathbb{B}_1$ is given in Lemma \ref{parabolic_pre_lem_chi} and $$\mathbb{B}_2:=L^{\bar{q}}(\mathbb{T}^1; W^{2, \bar{q}}(\mathbb{T}^d))\cap W^{1, \bar{q}}(\mathbb{T}^1; L^{\bar{q}}(\mathbb{T}^d)).$$
\end{lemma}

\begin{proof}
  To keep presentations simple, we suppress parameters $x, t$ and superscripts $\alpha, \beta$.

  For fixed $\ubar{i}, j$, observing that $B_{\ubar{i}j}(y, \tau)\in L^2(\mathbb{T}^{d+1})$ and $\int_{\mathbb{T}^{d+1}} B_{\ubar{i}j}(y, \tau) dyd\tau=0$, there exists $f_{\ubar{i}j}\in L^2(\mathbb{T}^1; H^2(\mathbb{T}^d))\cap H^1(\mathbb{T}^1; L^2(\mathbb{T}^d))$, solving the following parabolic problem in cell
  \begin{equation}\label{parabolic_flux-corr_eq_f}
    \begin{cases}
      \partial_\tau f_{\ubar{i}j}-\Delta_y f_{\ubar{i}j}=B_{\ubar{i}j}\quad \mathrm{in}~\mathbb{R}^{d+1},\\
     f_{\ubar{i}j} \textrm{ is 1-periodic in } (y, \tau),\textrm{ and }\int_{\mathbb{T}^{d+1}} f_{\ubar{i}j}=0,
    \end{cases}
  \end{equation}
where $\Delta_y=\sum_{i=1}^d\partial_{y_i}^2$. According to equations \eqref{parabolic_pre_eq_chi} and \eqref{parabolic_flux-corr_eq_f}, we know
\begin{align*}
  (\partial_\tau-\Delta_y)(\partial_if_{ij}+\partial_\tau f_{(d+1)j})=\partial_iB_{ij}+\partial_\tau B_{(d+1)j}=0\quad \mathrm{in~}\mathbb{R}^{d+1}.
\end{align*}
Therefore, $w=\partial_if_{ij}+\partial_\tau f_{(d+1)j}$ solves the parabolic problem
\begin{align*}
  \begin{cases}
    \partial_\tau w-\Delta_yw=0\quad\mathrm{in~}\mathbb{R}^{d+1},\\
  w \textrm{ is }1\textrm{-periodic in }(y, \tau) \textrm{ and }\int_{\mathbb{T}^{d+1}} w=0,
  \end{cases}
\end{align*}
which yields that
\begin{align}\label{parabolic_flux-corr_iden_f}
  \partial_if_{ij}+\partial_\tau f_{(d+1)j}=0.
\end{align}
For $1\leq j\leq d$, define
\begin{align*}\mathfrak{B}_{\ubar{k}\ubar{i}j}=
  \begin{cases}
    \partial_{\ubar{i}} f_{\ubar{k}j}-\partial_{\ubar{k}}f_{\ubar{i}j}, &\mathrm{if}~1\leq \ubar{k}, \ubar{i}\leq d,\\
    f_{\ubar{i}j}+\partial_{\ubar{i}}f_{(d+1)j}, &\mathrm{if}~\ubar{k}=d+1, 1\leq \ubar{i}\leq d,\\
    -f_{\ubar{k}j}-\partial_{\ubar{k}}f_{(d+1)j}, &\mathrm{if}~1\leq \ubar{k}\leq d, \ubar{i}=d+1,\\
    0, &\mathrm{if}~\ubar{k}=\ubar{i}=d+1.
  \end{cases}
\end{align*}
Obviously, $\mathfrak{B}_{\ubar{k}\ubar{i}j}=-\mathfrak{B}_{\ubar{i}\ubar{k}j}$.
Moreover, for $1\leq i, j\leq d$, by \eqref{parabolic_pre_eq_chi} and \eqref{parabolic_flux-corr_iden_f}, it is easy to verify that
\begin{align*} \partial_k\mathfrak{B}_{kij}+\partial_\tau\mathfrak{B}_{(d+1)ij}=\partial_k\partial_if_{kj}-\Delta_yf_{ij}+\partial_\tau f_{ij}+\partial_\tau\partial_i f_{(d+1)j}=B_{ij},
\end{align*}
and
\begin{align*}
  \partial_k\mathfrak{B}_{k(d+1)j}+\partial_\tau\mathfrak{B}_{(d+1)(d+1)j}=-\partial_kf_{kj}-\Delta_yf_{(d+1)j}=\partial_\tau f_{(d+1)j}-\Delta_yf_{(d+1)j}=B_{(d+1)j}.
\end{align*}

Lastly, the regularity estimates on $\mathfrak{B}$ follow from Meyers-type estimate for parabolic systems \cite{Campanato1980_parabolic,Armstrong2018_parabolic} as well as Lemma \ref{parabolic_pre_lem_chi} and Corollary \ref{parabolic_pre_coro_hatA}. We omit the details.
\end{proof}

\section{Smoothing operators and boundary layers}
\label{parabolic_sec_macr-smooth}

In this section, we establish several estimates for the macroscopic smoothing operators $S_\varepsilon$ in order to deal with the oscillating functions of the form $\phi(x, t; x/\varepsilon, t/\varepsilon^2)$ derived from multi-variable function $\phi(x, t; y, \tau)$ which is $1$-periodic in $(y, \tau)$. We also provide a result on temporal boundary layers. 

\subsection{Macroscopic smoothing operators for multi-variable functions}\label{parabolic_subsec_smoothing}

Suppose $\varphi_1\in C_0^\infty((-1/2, 1/2)), \varphi_2\in C_0^\infty(B(0, 1/2))$, where $B(0, 1/2)$ is the ball in $\mathbb{R}^d$ of radius $1/2$ centered at the origin $0$, such that
\begin{align}\label{parabolic_pre_smooth_cond_unit}\varphi_i\geq 0,~i=1,2\quad \mathrm{and}\quad \int_{\mathbb{R}}\varphi_1(s)~ds=1,\quad \int_{\mathbb{R}^d}\varphi_2(z)~dz=1.\end{align}For $\varepsilon>0$, let $\varphi_{1, \varepsilon}(s)=\varepsilon^{-2}\varphi_1(s/\varepsilon^2)$, $\varphi_{2, \varepsilon}(z)=\varepsilon^{-d}\varphi_2(z/\varepsilon)$. For a multi-variable function $\phi(x, t; y, \tau)$ on $\mathbb{R}^{d+1}\times \mathbb{R}^{d+1}$, set
\begin{align}
  \label{parabolic_pre_def_Si}
  \begin{split}
    S^t_\varepsilon(\phi)(x, t; y, \tau)&:=\int_{\mathbb{R}}\phi(x, t-s; y, \tau)\varphi_{1, \varepsilon}(s) ds,\\
S^x_\varepsilon(\phi)(x, t; y, \tau)&:=\int_{\mathbb{R}^d}\phi(x-z, t; y, \tau)\varphi_{2, \varepsilon}(z) dz,
  \end{split}
\end{align}
which are smoothing operators w.r.t. $t$ and $x$, respectively. In the following, we use the abbreviation
\begin{align*}%\label{parabolic_pre_def_S}
S_\varepsilon:=S^t_\varepsilon\circ S^x_\varepsilon=S^x_\varepsilon\circ S^t_\varepsilon,
\end{align*}
and more briefly,
\begin{align}\label{parabolic_conver_def_tilde}\widetilde{\phi}:=S_\varepsilon(\phi).\end{align}
Furthermore, we may omit the subscript $\varepsilon$ when $\varepsilon=1$. We remark that $S_\varepsilon$ is commutative with all the partial derivatives w.r.t. $x$, $t$, $y$ and $\tau$. 

Before stating the properties of $S_\varepsilon$, we introduce the notations
\begin{align*}
  \Omega_\delta:=\{x\in \Omega: \textrm{dist}(x, \partial\Omega)<\delta\}
\end{align*}
for the spatial boundary layer, and
\begin{align*}
    \Omega_{T}^{k, \varepsilon}:=[\Omega_{k\varepsilon}\times(0, T)]\cup[\Omega\times(0, k\varepsilon^2)]\cup[\Omega\times(T-k\varepsilon^2, T)]
\end{align*}
for the spatiotemporal boundary layer in $\Omega_T$, where $\delta, k>0$.

\begin{lemma}\label{parabolic_pre_smooth_lem_1}
  Let $T>0$ and $h(\tau): \mathbb{R}\rightarrow \mathbb{R}$ be $1$-periodic. Then for any $\varepsilon\leq \sqrt{T}$,
    \begin{align}\label{parabolic_conver_es_smoothing_h}
    \int_0^T|h(\frac{t}{\varepsilon^2})|dt\leq CT\int_{\mathbb{T}^1}|h(\tau)|d\tau,
  \end{align}
  where $C$ is a constant.

  Moreover, suppose $g(t; \tau)\in L^{q_1}(\mathbb{R}; L^{s_1}(\mathbb{T}^1))$, $f(t, \tau)\in L^{q_2}(\mathbb{R}; L^{s_2}(\mathbb{T}^1))$ with $1\leq q<\infty$, $\frac{1}{q_1}+\frac{1}{q_2}=\frac{1}{s_1}+\frac{1}{s_2}=\frac{1}{q}$. We have for $\varepsilon>0$,
  \begin{align}
    \label{parabolic_conver_es_smoothing_gf}
    \|[S^t_\varepsilon(g)\cdot S^t_\varepsilon(f)]^{\varepsilon}\|_{L^q(\varepsilon^2, T-\varepsilon^2)}\leq C\|g\|_{L^{q_1}(0, T; L^{s_1}(\mathbb{T}^1))}\|f\|_{L^{q_2}(0, T; L^{s_2}(\mathbb{T}^1))}.
  \end{align}
\end{lemma}

\begin{proof}
  Estimate \eqref{parabolic_conver_es_smoothing_h} follows from the change of variables. For \eqref{parabolic_conver_es_smoothing_gf}, by considering $g\cdot \mathbf{1}_{(0, T)\times \mathbb{R}}$ and $f\cdot \mathbf{1}_{(0, T)\times \mathbb{R}}$ instead, it is sufficient to prove the whole space case. By rescaling, we also suppose that $\varepsilon=1$ and omit the subscript $\varepsilon$.

Set $G(t_1, t_2; \tau)=g(t_1; \tau)\cdot f(t_2; \tau)$ and denote
  \begin{align}\label{notation_tense}
    S^{t_1}\otimes S^{t_2}(G)(t, \tau):=\int_{\mathbb{R}}\int_{\mathbb{R}}G(\varsigma_1, \varsigma_2; \tau)\varphi_1(t-\varsigma_1)\varphi_1(t-\varsigma_2)d\varsigma_2 d\varsigma_1. 
  \end{align}
  Then $S^t(g)(t; \tau)\cdot S^t(f)(t; \tau)=S^{t_1}\otimes S^{t_2}(G)(t; \tau)$. By Fubini's theorem and H\"{o}lder's inequality,
  \begin{align}\label{parabolic_conver_smoothing_es_H}
    \begin{split}
    \int_{\mathbb{R}}|S^{t_1}\otimes S^{t_2}(G)(t; t)|^qdt&\leq C\int_{\mathbb{R}}\fint_{|\varsigma_1-t|\leq \frac{1}{2}}\fint_{|\varsigma_2-t|\leq \frac{1}{2}}|G(\varsigma_1, \varsigma_2; t)|^qd\varsigma_2 d\varsigma_1 dt\\&\leq C\int_{\mathbb{R}}\int_{|\varsigma_1-\varsigma_2|\leq 1}\int_{|\varsigma_2-t|\leq \frac{1}{2}}|G(\varsigma_1, \varsigma_2; t)|^qdtd\varsigma_1 d\varsigma_2\\&\leq C\int_{\mathbb{R}}\int_{|\varsigma_1-\varsigma_2|\leq 1}\|G(\varsigma_1, \varsigma_2; \cdot)\|_{L^{q}(\mathbb{T}^1)}^qd\varsigma_1d\varsigma_2\\&\leq C\|g\|_{L^{q_1}(\mathbb{R}; L^{s_1}(\mathbb{T}^1))}^q\|f\|_{L^{q_2}(\mathbb{R}; L^{s_2}(\mathbb{T}^1))}^q, 
  \end{split}
  \end{align}
  where we have used \eqref{parabolic_conver_es_smoothing_h} in the third inequality and H\"{o}lder's inequality in the last step. The proof is completed. 
  \end{proof}

\begin{lemma}\label{parabolic_pre_smooth_lem_2}
  Let $\Omega$ be a Lipschitz domain in $\mathbb{R}^d$, $T>0$ and $1\leq p, q<\infty$. Suppose $g\in L^{q_1, p_1}(\mathbb{R}^{d+1}; \bm{L^{s_1, r_1}})$, $f\in L^{q_2, p_2}(\mathbb{R}^{d+1}; \bm{L^{s_2, r_2}})$ with $\frac{1}{p_1}+\frac{1}{p_2}=\frac{1}{r_1}+\frac{1}{r_2}=\frac{1}{p}$, $\frac{1}{q_1}+\frac{1}{q_2}=\frac{1}{q}$. Let $\frac{1}{s}=\frac{1}{s_1}+\frac{1}{s_2}$. If $s\geq \max\{p, q\}$, then for $\varepsilon>0$,
  $$\|[S_\varepsilon(g)\cdot S_{\varepsilon}(f)]^\varepsilon\|_{L^{q, p}(\Omega_T\setminus\Omega_T^{1, \varepsilon})}\leq C\|g\|_{L^{q_1, p_1}(\Omega_T; \bm{L^{s_1, r_1}})}\|f\|_{L^{q_2, p_2}(\Omega_T; \bm{L^{s_2, r_2}})},$$
  where $C$ depends only on $d$.

  Specially, if $f\equiv 1$ on $\Omega_T$ and $s\geq \max\{p, q\}$, we have
  $$\|[S_\varepsilon(g)]^\varepsilon\|_{L^{q, p}(\Omega_T\setminus\Omega_T^{1, \varepsilon})}\leq C\|g\|_{L^{q, p}(\Omega_T; \bm{L^{s, p}})}.$$
\end{lemma}

\begin{proof}
  By rescaling and introducing invisible characteristic functions, it is sufficient to prove the case where $\varepsilon=1$ and $\Omega_T=\mathbb{R}^{d+1}$. For each $t$ fixed, it follows from Fubini's theorem and H\"{o}lder's inequality that
  \begin{align*}
    \|[S(g)S(f)](x, t; x, t)\|_{L^p_x(\mathbb{R}^d)}&\leq C\Big(\int_{\mathbb{R}^d}\int_{|\omega-\xi|\leq 1}\int_{|\xi-x|\leq \frac{1}{2}}|S^t(g)(\omega, t; x, t)\cdot S^t(f)(\xi, t; x, t)|^pdxd\omega d\xi\Big)^{\frac{1}{p}}\\&\leq C\Big(\int_{\mathbb{R}^d}\int_{|\omega-\xi|\leq 1}\|S^t(g)\|_{L^{r_1}_y}^p(\omega, t; t)d\omega\cdot \|S^t(f)\|_{L^{r_2}_y}^p(\xi, t; t)d\xi\Big)^{\frac{1}{p}}\\&\leq CS^{t_1}\otimes S^{t_2}(G)(t; t),
\end{align*}
where we have applied Minkowski's inequality to take the integrals of $S^t$ outside in the last inequality and notations \eqref{parabolic_pre_notation_multiple} and \eqref{notation_tense} were used with $$G(\varsigma_1, \varsigma_2; \tau)=\Big(\int_{\mathbb{R}^d}\int_{|\omega-\xi|\leq 1}\|g\|_{L^{r_1}_y}^p(\omega, \varsigma_1; \tau)d\omega\cdot \|f\|_{L^{r_2}_y}^p(\xi, \varsigma_2; \tau) d\xi\Big)^{1/p}.$$
Then calculations similar to \eqref{parabolic_conver_smoothing_es_H} imply that
\begin{align}\label{parabolic_macr-smooth_es_gf_1}
  \begin{split}
    \|[S(g)S(f)]^1\|_{L^{q, p}(\mathbb{R}^{d+1})}&\leq C\|S^{t_1}\otimes S^{t_2}(G)(t; t)\|_{L^q(\mathbb{R})}\\&\leq C\Big(\int_{\mathbb{R}}\int_{|\varsigma_1-\varsigma_2|\leq 1}\|G(\varsigma_1, \varsigma_2; \cdot)\|_{L^q(\mathbb{T}^1)}^qd\varsigma_1 d\varsigma_2\Big)^{1/q}.
    \end{split}
  \end{align}
By Minkowski's inequality and H\"{o}lder's inequality, it is not hard to see that for $s\geq \max\{p, q\}$
  \begin{align*}
    \|G(\varsigma_1, \varsigma_2; \cdot)\|_{L^{q}(\mathbb{T}^1)}\leq \|G(\varsigma_1, \varsigma_2; \cdot)\|_{L^{s}(\mathbb{T}^1)}\leq \|g\|_{L^{p_1, s_1, r_1}_{x, \tau, y}(\mathbb{R}^d)}(\varsigma_1)\cdot \|f\|_{L^{p_2, s_2, r_2}_{x, \tau, y}(\mathbb{R}^d)}(\varsigma_2).
  \end{align*}
Substituting the last inequality into \eqref{parabolic_macr-smooth_es_gf_1} and applying H\"{o}lder's inequality, we complete the proof.
\end{proof}

\begin{remark}\label{parabolic_pre_remark_1}
  Lemma \ref{parabolic_pre_smooth_lem_2} continues to hold even if $\varphi_1, \varphi_2$ do not satisfy \eqref{parabolic_pre_smooth_cond_unit}. In particular, we have for $s\geq \max\{p, q\}$
  \begin{align*}
    \|[\nabla_x S_\varepsilon(h)]^\varepsilon\|_{L^{q, p}(\Omega_T\setminus\Omega_T^{1, \varepsilon})}=\varepsilon^{-1}\|[(\nabla\varphi_2)_\varepsilon* S^t_\varepsilon(h)]^\varepsilon\|_{L^{q, p}(\Omega_T\setminus\Omega_T^{1, \varepsilon})}\leq C\varepsilon^{-1}\|h\|_{L^{q, p}(\Omega_T; \bm{L^{s, p}})},
  \end{align*}
  where $(\nabla\varphi_2)_\varepsilon(x)=\varepsilon^{-d}(\nabla\varphi_2)(x/\varepsilon)$.
\end{remark}

\begin{lemma}\label{parabolic_pre_smooth_lem_3}
  Let $\Omega$ be a Lipschitz domain in $\mathbb{R}^d$, $T>0$ and $1\leq p, q<\infty$. Suppose that $\nabla_xg\in L^{q_1, p_1}(\Omega_T; \bm{L^\infty})$, $g\in W^{\frac{1}{2}, \check{q}_1}(0, T; L^{\check{p}_1}(\Omega; \bm{L^\infty}))$, $f\in L^{q_2, p_2}(\Omega_T; \bm{L^{s_2, r_2}})\cap L^{\check{q}_2, \check{p}_2}(\Omega_T; \bm{L^{s_2, r_2}})$, $h\in L^{q_3, p_3}(\Omega_T; \bm{L^{s_3, r_3}})\cap L^{\check{q}_3, \check{p}_3}(\Omega_T; \bm{L^{s_3, r_3}})$, where $p_i, \check{p}_i, q_i, \check{q}_i$ satisfy $\frac{1}{p}=\frac{1}{p_1}+\frac{1}{p_2}+\frac{1}{p_3}=\frac{1}{\check{p}_1}+\frac{1}{\check{p}_2}+\frac{1}{\check{p}_3}$, $\frac{1}{q}=\frac{1}{q_1}+\frac{1}{q_2}+\frac{1}{q_3}=\frac{1}{\check{q}_1}+\frac{1}{\check{q}_2}+\frac{1}{\check{q}_3}$. Let $\frac{1}{r}=\frac{1}{r_2}+\frac{1}{r_3}$, $\frac{1}{s}=\frac{1}{s_2}+\frac{1}{s_3}$. If $r\geq p$ and $s\geq \max\{p, q\}$, then for $\varepsilon>0$,
  \begin{align*}%\label{parabolic_pre_smooth_es_3}
    \begin{split}
      &\quad\|\{[S_\varepsilon(g)\cdot S_\varepsilon(f)-S_\varepsilon(g\cdot f)]\cdot S_\varepsilon(h)\}^\varepsilon\|_{L^{q, p}(\Omega_T\setminus\Omega_T^{1, \varepsilon})}\\&\leq C\varepsilon\{\|\nabla_xg\|_{L^{q_1, p_1}(\Omega_T; \bm{L^\infty})}\cdot\|f\|_{L^{q_2, p_2}(\Omega_T; \bm{L^{s_2, r_2}})}\cdot\|h\|_{L^{q_3, p_3}(\Omega_T; \bm{L^{s_3, r_3}})}\\&\qquad+[g]_{W^{\frac{1}{2}, \check{q}_1}(0, T; L^{\check{p}_1}(\Omega; \bm{L^\infty}))}\cdot\|f\|_{L^{\check{q}_2, \check{p}_2}(\Omega_T; \bm{L^{s_2, r_2}})}\cdot\|h\|_{L^{\check{q}_3, \check{p}_3}(\Omega_T; \bm{L^{s_3, r_3}})}\},
    \end{split}
\end{align*}
  where $[g]_{W^{\frac{1}{2}, \check{q}_1}(0, T; L^{\check{p}_1}(\Omega; \bm{L^\infty}))}$ is given by \eqref{parabolic_pre_def_seminorm} and $C$ depends only on $d$.
  \end{lemma}
  \begin{proof}
    To keep the process simple, we prove the case where $\Omega_T=\mathbb{R}^{d+1}$ and $h\equiv 1$. The proof of the general case is similar. Without loss of generality, we also suppose that $\varepsilon=1$ by rescaling and $g$ is independent of $(y, \tau)$. Observe that
    \begin{align*}
      \|[S(g)\cdot S(f)-S(g\cdot f)]^1\|_{L^{q, p}(\mathbb{R}^{d+1})}&\leq \|[S^x(S^t(g))\cdot S^x(S^t(f))-S^x(S^t(g)\cdot S^t(f))]^1\|_{L^{q, p}(\mathbb{R}^{d+1})}\\&\qquad+\|[S^x(S^t(g)\cdot S^t(f)-S^t(g\cdot f))]^1\|_{L^{q, p}(\mathbb{R}^{d+1})},
    \end{align*}
where the superscript $1$ has the same meaning as \eqref{notation_sup_ve}. 
    
  For the first term, let $H\in L^{p'}(\mathbb{R}^{d})$. By using the Poincar\'{e}'s inequality
  \begin{align}\label{parabolic_ineq_poincare}
    \int_{B(x, r)}|u(y)-u(z)|dy\leq Cr^{d}\int_{B(x, r)}|\nabla u(y)||y-z|^{1-d}dy
  \end{align}
for any $u\in C^1(B(x, r))$ and $z\in B(x, r)$, one can obtain
\begin{align*}
  &\quad\Big|\int_{\mathbb{R}^{d}}[S^x(S^t(g))\cdot S^x(S^t(f))-S^x(S^t(g)\cdot S^t(f))](x, t; x, t)\cdot H(x) dx\Big|\\&\leq C\int_{\mathbb{R}^d}\int_{|\xi-x|\leq \frac{1}{2}}|S^t(f)(\xi, t; x, t)|\int_{|\omega-x|\leq \frac{1}{2}}|S^t(\nabla_x g)(\omega, t)||\omega-\xi|^{1-d}d\omega d\xi\cdot |H(x)| dx\\&\leq C\int_{\mathbb{R}^d}\int_{|\omega-\xi|\leq 1}|S^t(\nabla_x g)(\omega, t)||\omega-\xi|^{1-d}d\omega\cdot \int_{|\xi-x|\leq \frac{1}{2}}|S^t(f)(\xi, t; x, t)||H(x)|dx d\xi\\&\leq C\int_{\mathbb{R}^d}\int_{\mathbb{R}^d}|S^t(\nabla_x g)(\omega, t)|\phi(\omega-\xi)d\omega\cdot\|S^t(f)\|_{L^p_y}(\xi, t; t)\|H\|_{L^{p'}(B(\xi, \frac{1}{2}))} d\xi\\&\leq C\Big(\int_{\mathbb{R}^d}\int_{\mathbb{R}^d}|S^t(\nabla_x g)(\omega, t)|^p\phi(\omega-\xi)\|S^t(f)\|_{L^p_y}^p(\xi, t; t)d\omega d\xi\Big)^{1/p}\|H\|_{L^{p'}(\mathbb{R}^d)},
\end{align*}
where Fubini's theorem as well as H\"{o}lder's inequality was used and $\phi(z)=|z|^{1-d}\mathbf{1}_{\{|z|\leq 1\}}$. Since $H\in L^{p'}(\mathbb{R}^d)$ is arbitrary, this implies that
\begin{align*}
  &\quad\|[S^x(S^t(g))\cdot S^x(S^t(f))-S^x(S^t(g)\cdot S^t(f))](x, t; x, t)\|_{L^p_x(\mathbb{R}^d)}\\&\leq C\Big(\int_{\mathbb{R}^d}\int_{\mathbb{R}^d}|S^t(\nabla_x g)(\omega, t)|^p\phi(\omega-\xi)\|S^t(f)\|_{L^p_y}^p(\xi, t; t)d\omega d\xi\Big)^{1/p}\\&\leq CS^{t_1}\otimes S^{t_2}(G)(t; t),
\end{align*}
where we have applied Minkowski's inequality to take the integrals of $S^t$ outside and
\begin{align*}
  G(\varsigma_1, \varsigma_2; \tau)=\Big(\int_{\mathbb{R}^d}\int_{\mathbb{R}^d}|\nabla_x g(\omega, \varsigma_1)|^p\phi(\omega-\xi)\|f\|_{L^p_y}^p(\xi, \varsigma_2; \tau)d\omega d\xi\Big)^{1/p}. 
\end{align*}
Similar to the arguments of Lemma \ref{parabolic_pre_smooth_lem_2}, we obtain
\begin{align*}
     \|[S^x(S^t(g)\cdot S^t(f)-S^t(g\cdot f))]^1\|_{L^{q, p}(\mathbb{R}^{d+1})}&\leq C\|S^{t_1}\otimes S^{t_2}(G)(t; t)\|_{L^q(\mathbb{R})}\\&\leq C\|\nabla_x g\|_{L^{q_1, p_1}_{t, x}(\mathbb{R}^{d+1})}\|f\|_{L^{q_2, p_2, s, r}_{t, x, \tau, y}(\mathbb{R}^{d+1})}
\end{align*}
for $s\geq \max\{p, q\}$, where Young's inequality is used in the second step and $C$ depends only on $d$.

To deal with the second term, note that
    \begin{align*}
      &\quad|S^t(g)\cdot S^t(f)-S^t(g\cdot f)|(x, t; y, \tau)\\&\leq C\int_{\mathbb{R}}|f(x, \varsigma_1; y, \tau)|\cdot\varphi_{1}(t-\varsigma_1)\cdot\fint_{|\varsigma_2-t|\leq \frac{1}{2}}|g(x, \varsigma_1)-g(x, \varsigma_2)|d\varsigma_2 d\varsigma_1\\&\leq C\int_{\mathbb{R}}|f(x, \varsigma_1; y, \tau)|\cdot\varphi_{1}(t-\varsigma_1)\cdot\fint_{|\varsigma_2-\varsigma_1|\leq 1}|g(x, \varsigma_1)-g(x, \varsigma_2)|d\varsigma_2 d\varsigma_1\\&= CS^t(|f|\cdot G),
    \end{align*}
where $$G(x, t)=\fint_{|\varsigma-t|\leq 1}|g(x, t)-g(x, \varsigma)|d\varsigma.$$ Therefore, it follows from Lemma \ref{parabolic_pre_smooth_lem_2} that, for $s\geq \max\{p, q\}$,
\begin{align*}
      \|[S^x(S^t(g)\cdot S^t(f)-S^t(g\cdot f))]^1\|_{L^{q, p}(\mathbb{R}^{d+1})}&\leq C\|[S(|f|\cdot G)]^1\|_{L^{q, p}(\mathbb{R}^{d+1})}\\&\leq C\|G\|_{L^{\check{q}_1, \check{p}_1}(\mathbb{R}^{d+1})}\|f\|_{L^{\check{q}_2, \check{p}_2}(\mathbb{R}^{d+1}; \bm{L^{s, p}})}.
\end{align*}
Moreover, by H\"{o}lder's inequality,
\begin{align*}
  \|G\|_{L^{\check{q}_1, \check{p}_1}(\mathbb{R}^{d+1})}&\leq\Big(\int_{\mathbb{R}}\fint_{|\varsigma-t|\leq 1}\|g(\cdot, t)-g(\cdot, \varsigma)\|_{L^{\check{p}_1}(\mathbb{R}^d)}^{\check{q}_1}d\varsigma dt\Big)^{1/\check{q}_1}\\&\leq \Big(\int_{\mathbb{R}}\int_{|\varsigma-t|\leq 1}\frac{\|g(\cdot, t)-g(\cdot, \varsigma)\|_{L^{\check{p}_1}(\mathbb{R}^d)}^{\check{q}_1}}{|t-\varsigma|^{1+{\check{q}_1}/2}}d\varsigma dt\Big)^{1/\check{q}_1}\\&\leq [g]_{W^{\frac{1}{2}, \check{q}_1}(\mathbb{R}; L^{\check{p}_1}(\mathbb{R}^d))},
\end{align*}
which implies that
\begin{align*}
  \|[S^x(S^t(g)\cdot S^t(f)-S^t(g\cdot f))]^1\|_{L^{q, p}(\mathbb{R}^{d+1})}\leq C[g]_{W^{\frac{1}{2}, \check{q}_1}(\mathbb{R}; L^{\check{p}_1}(\mathbb{R}^d))}\|f\|_{L^{\check{q}_2, \check{p}_2}(\mathbb{R}^{d+1}; \bm{L^{s, r}})}.
\end{align*}

By combining the estimates of these two terms above, we accomplish the proof.
\end{proof}

\begin{lemma}\label{parabolic_pre_smooth_lem_4}
  Let the assumptions of Lemma \ref{parabolic_pre_smooth_lem_3} hold. Suppose further $p_1>1$, $\check{p}_1=\infty$, $r\geq p$ and $\frac{1}{r}\leq \frac{1}{d}+\frac{1}{p_2}+\frac{1}{p_3}$, $s\geq \max\{p, q\}$ and $\frac{1}{s}<\frac{1}{2}+\frac{1}{\check{q}_2}+\frac{1}{\check{q}_3}$. Then for $\varepsilon>0$,
  \begin{align*}&\quad\|[(g-S_\varepsilon(g))\cdot S_\varepsilon(f)\cdot S_\varepsilon(h)]^\varepsilon\|_{L^{q, p}(\Omega_T\setminus\Omega_T^{1, \varepsilon})}\\&\leq C\varepsilon\{\|\nabla_xg\|_{L^{q_1, p_1}(\Omega_T; \bm{L^\infty})}\cdot\|f\|_{L^{q_2, p_2}(\Omega_T; \bm{L^{s_2, r_2}})}\cdot\|h\|_{L^{q_3, p_3}(\Omega_T; \bm{L^{s_3, r_3}})}\\&\qquad+[g]_{W^{\frac{1}{2}, \check{q}_1}(0, T; L^{\infty}(\Omega; \bm{L^\infty}))}\cdot\|f\|_{L^{\check{q}_2, \check{p}_2}(\Omega_T; \bm{L^{s_2, r_2}})}\cdot\|h\|_{L^{\check{q}_3, \check{p}_3}(\Omega_T; \bm{L^{s_3, r_3}})}\},\end{align*}
  where $C$ depends only on $d, q, p_1, \check{q}_1, s$.
\end{lemma}
\begin{proof}
  As in the proof of Lemma \ref{parabolic_pre_smooth_lem_3}, we prove the case where $\Omega_T=\mathbb{R}^{d+1}$ and $h\equiv 1$, in which case $p_3=\check{p}_3=q_3=\check{q}_3=\infty$ and $\frac{1}{p}=\frac{1}{p_1}+\frac{1}{p_2}=\frac{1}{\check{p}_2}$, $\frac{1}{q}=\frac{1}{q_1}+\frac{1}{q_2}=\frac{1}{\check{q}_1}+\frac{1}{\check{q}_2}$. Without loss of generality, we also suppose that $\varepsilon=1$ and $g$ is independent of $(y, \tau)$. Note that
  \begin{align}\label{parabolic_ineq_triangle}
    \begin{split}
          &\|[(g-S(g))\cdot S(f)]^1\|_{L^{q, p}(\mathbb{R}^{d+1})}\leq \|[(g-S^t(g))\cdot S(f)]^1\|_{L^{q, p}(\mathbb{R}^{d+1})}\\&\qquad\qquad+\|[S^t(g-S^x(g))\cdot S(f)]^1\|_{L^{q, p}(\mathbb{R}^{d+1})}.
    \end{split}
  \end{align}

  For the second term, we focus on the case $p\leq d$, as the case $p>d$ is simple (see \cite{Xu2019_stratified2}). Let $H\in L^{p'}(\mathbb{R}^{d})$, where $p'=\frac{p}{p-1}$. For each $t$ fixed, it follows from inequality \eqref{parabolic_ineq_poincare} and Fubini's theorem
  \begin{align*}
    &\quad\Big|\int_{\mathbb{R}^{d}}S^t(g-S^x(g))(x, t)\cdot S(f)(x, t; x, t)\cdot H(x)dx\Big|\\&\leq C\int_{\mathbb{R}^d}\int_{|\xi-x|\leq \frac{1}{2}}\int_{|\omega-x|\leq \frac{1}{2}}|S^t(\nabla_xg)(\omega, t)||\omega-x|^{1-d}d\omega\cdot |S^t(f)(\xi, t; x, t)H(x)| d\xi dx\\&\leq C\int_{\mathbb{R}^d}d\xi\cdot\int_{\mathbb{R}^d}\int_{\mathbb{R}^d}|S^t(\nabla_xg)(\omega, t)|\mathbf{1}_{|\omega-\xi|\leq 1}\cdot|\omega-x|^{1-d}\cdot|S^t(f)(\xi, t; x, t)H(x)|\mathbf{1}_{|x-\xi|\leq 1} dx d\omega\\&\leq C\int_{\mathbb{R}^d}\|S^t(\nabla_xg)(\cdot, t)\|_{L^{p_1}(B(\xi, 1))}\|S^t(f)(\xi, t; \cdot, t)H(\cdot)\|_{L^P(B(\xi, 1))}d\xi\\&\leq C\int_{\mathbb{R}^d}\|S^t(\nabla_xg)(\cdot, t)\|_{L^{p_1}(B(\xi, 1))}\cdot\|S^t(f)\|_{L^r_y}(\xi, t; t)\cdot\|H\|_{L^{p'}(B(\xi, 2))}d\xi\\&\leq C\Big(\int_{\mathbb{R}^d}\|S^t(\nabla_xg)(\cdot, t)\|_{L^{p_1}(B(\xi, 1))}^p\cdot\|S^t(f)\|_{L^r_y}^p(\xi, t; t)d\xi\Big)^{1/p}\cdot\|H\|_{L^{p'}(\mathbb{R}^d)},
\end{align*}
where $\frac{1}{P}=\frac{1}{d}+\frac{1}{p_1'}=\frac{1}{d}+\frac{1}{p_2}+\frac{1}{p'}\geq\frac{1}{r}+\frac{1}{p'}$ and we have used the Hardy-Littlewood-Sobolev inequality as well as H\"{o}lder's inequality. By duality, this leads to
\begin{align*}
  \|[S^t(g-S^x(g))\cdot S(f)]^1\|_{L^p_x(\mathbb{R}^{d})}&\leq C\Big(\int_{\mathbb{R}^d}\|S^t(\nabla_xg)(\cdot, t)\|_{L^{p_1}(B(\xi, 1))}^p\cdot\|S^t(f)\|_{L^r_y}^p(\xi, t; t)d\xi\Big)^{1/p}\\&\leq CS^{t_1}\otimes S^{t_2}(G)(t; t),
\end{align*}
where
\begin{align*}
  G(\varsigma_1, \varsigma_2; \tau)=\Big(\int_{\mathbb{R}^d}\|\nabla_xg(\cdot, \varsigma_1)\|_{L^{p_1}(B(\xi, 1))}^p\cdot\|f\|_{L^r_y}^p(\xi, \varsigma_2; \tau)d\xi\Big)^{1/p}.
\end{align*}
Following the arguments of Lemma \ref{parabolic_pre_smooth_lem_2}, we conclude
\begin{align}\label{parabolic_conver_smoothing_es_term1}
  \begin{split}
     \|[S^t(g-S^x(g))\cdot S(f)]^1\|_{L^{q, p}(\mathbb{R}^{d+1})}&\leq C\|S^{t_1}\otimes S^{t_2}(G)(t; t)\|_{L^q(\mathbb{R})}\\&\leq C\|\nabla_x g\|_{L^{q_1, p_1}_{t, x}(\mathbb{R}^{d+1})}\|f\|_{L^{q_2, p_2, s, r}_{t, x, \tau, y}(\mathbb{R}^{d+1})}
  \end{split}
\end{align}
for $s\geq \max\{p, q\}$. This gives the estimate for the second term.

To deal with the first term in the r.h.s. of \eqref{parabolic_ineq_triangle}, we prove a special case of the desired estimate, that is,
\begin{align}\label{parabolic_es_special}
  \|(g-S^t(g))\cdot [S^t(f)]^1\|_{L^q(\mathbb{R})}\leq C[g]_{W^{\frac{1}{2}, \check{q}_1}(\mathbb{R})}\|f\|_{L^{\check{q}_2}(\mathbb{R}; L^s(\mathbb{T}))},
\end{align}
where $g$ is independent of $x, y, \tau$ and $f$ is independent of $x, y$. The vector-valued case can be proved in the same manner. Since $g$ can be approximated by smooth functions, we just suppose $g$ itself is smooth, in which case we write
\begin{align*} g(t)-S^t(g)(t)&=\int_{\mathbb{R}}(g(t)-g(t-\varsigma))\varphi_1(\varsigma)d\varsigma=\int_{\mathbb{R}}\int_0^1\partial_tg(t-\theta \varsigma)\cdot\varsigma d\theta \varphi_1(\varsigma)d\varsigma\\&=-\int_0^1\int_{\mathbb{R}}\partial_tg(\xi)\cdot\varPhi_\theta(t-\xi)d\xi d\theta,
\end{align*}
where $\varPhi(\varsigma)=\varsigma\varphi_1(\varsigma)$ and $\varPhi_\theta(\cdot)=\frac{1}{\theta}\varPhi(\frac{\cdot}{\theta})$. For $H\in L^{q'}(\mathbb{R})$, it holds that
\begin{align*}
  &\quad\int_{\mathbb{R}}(g-S^t(g))\cdot S^t(f)(t; t)\cdot H(t)dt\\&=-\int_0^1\iiint_{\mathbb{R}\times\mathbb{R}\times\mathbb{R}}\partial_tg(\xi)\cdot\varPhi_\theta(t-\xi)d\xi f(\varsigma; t)H(t)\varphi_1(t-\varsigma)d\varsigma dtd\theta\\&=-\int_0^1\iiint_{\mathbb{R}\times\mathbb{R}\times\mathbb{R}}g(\xi)\cdot\partial_t\varPhi_\theta(t-\xi)d\xi f(\varsigma; t)H(t)\varphi_1(t-\varsigma)dtd\varsigma d\theta.
\end{align*}
For each $\theta$ and $\varsigma$ fixed, we can cover the interval $(\varsigma-\frac{1}{2}, \varsigma+\frac{1}{2})$ by finite intervals $[t_i, t_{i+1}]$ of length $\theta$ and obtain
\begin{align*}
  &\quad\Big|\int_{\mathbb{R}}\int_{\mathbb{R}}g(\xi)\partial_t\varPhi_\theta(t-\xi)f(\varsigma; t)H(t)\varphi_1(t-\varsigma)d\xi dt\Big|\\&=\Big|\sum_{i}\int_{t_i}^{t_{i+1}}\int_{\mathbb{R}}(g(\xi)-\fint_{t_i}^{t_{i+1}}g)\partial_t\varPhi_\theta(t-\xi)f(\varsigma; t)H(t)\varphi_1(t-\varsigma)d\xi dt\Big|\\&\leq \sum_{i}\int_{\mathbb{R}}\int_{t_i}^{t_{i+1}}\fint_{t_i}^{t_{i+1}}\frac{|g(\xi)-g(\zeta)|}{|\xi-\zeta|^{1/2+1/\check{q}_1}}d\zeta\cdot\theta^{\frac{1}{2}+\frac{1}{\check{q}_1}}\cdot|\partial_t\varPhi_\theta(t-\xi)|\cdot|f(\varsigma; t)H(t)|dtd\xi\\&\leq \theta^{-\frac{1}{2}}\sum_{i}\int_{\mathbb{R}}\int_{t_i}^{t_{i+1}}\Big(\int_{t_i}^{t_{i+1}}\frac{|g(\xi)-g(\zeta)|^{\check{q}_1}}{|\xi-\zeta|^{\check{q}_1/2+1}}d\zeta\Big)^{1/\check{q}_1}\cdot|(\partial_t\varPhi)_\theta(t-\xi)|\cdot|f(\varsigma; t)H(t)|dtd\xi\\&\leq \theta^{-\frac{1}{2}}\int_{\mathbb{R}}\int_{\varsigma-1}^{\varsigma+1}\Big(\int_{\varsigma-1}^{\varsigma+1}\frac{|g(\xi)-g(\zeta)|^{\check{q}_1}}{|\xi-\zeta|^{\check{q}_1/2+1}}d\zeta\Big)^{1/\check{q}_1}\cdot|(\partial_t\varPhi)_\theta(t-\xi)|\cdot|f(\varsigma; t)H(t)|dtd\xi\\&\leq \theta^{-\frac{1}{2}}\Big(\int_{\mathbb{R}}\int_{\varsigma-1}^{\varsigma+1}\frac{|g(\xi)-g(\zeta)|^{\check{q}_1}}{|\xi-\zeta|^{\check{q}_1/2+1}}d\zeta d\xi\Big)^{1/\check{q}_1}\|(\partial_t\varPhi)_\theta\|_{L^P}\|f(\varsigma; \cdot)H(\cdot)\|_{L^{\frac{sq'}{s+q'}}(\varsigma-1, \varsigma+1)}\\&\leq C\theta^{-\frac{1}{2}+\frac{1}{P}-1}\Big(\int_{\mathbb{R}}\int_{\varsigma-1}^{\varsigma+1}\frac{|g(\xi)-g(\zeta)|^{\check{q}_1}}{|\xi-\zeta|^{\check{q}_1/2+1}}d\zeta d\xi\Big)^{1/\check{q}_1}\cdot\|f(\varsigma; \cdot)\|_{L^s(\mathbb{T})}\cdot\|H\|_{L^{q'}(\varsigma-1, \varsigma+1)},
\end{align*}
where $(\partial_t\varPhi)_\theta(\cdot)=\frac{1}{\theta}\partial_t\varPhi(\frac{\cdot}{\theta})$ and we have used Young's inequality with $\frac{1}{\check{q}_1}+\frac{1}{P}+\frac{1}{s}+\frac{1}{q'}=2$ in the fifth step. Note that $P<2$. Thus, integrating w.r.t. $\theta$ and $\varsigma$ and applying H\"{o}lder's inequality, we conclude
\begin{align*}
  \Big|\int_{\mathbb{R}}(g-S^t(g))\cdot S^t(f)(t; t)\cdot H(t)dt\Big|\leq C[g]_{W^{\frac{1}{2}, \check{q}_1}(\mathbb{R})}\|f\|_{L^{\check{q}_2}(\mathbb{R} ;L^s(\mathbb{T}))}\|H\|_{L^{q'}(\mathbb{R})},
\end{align*}
which gives exactly \eqref{parabolic_es_special}.

The desired result follows from \eqref{parabolic_ineq_triangle} and \eqref{parabolic_conver_smoothing_es_term1}--\eqref{parabolic_es_special}.
\end{proof}

\begin{corollary}\label{parabolic_coro_smoothing_main}
  Suppose the assumptions of Lemma \ref{parabolic_pre_smooth_lem_4} hold and $\psi\in C^\infty_c(\Omega_T)$. Then for $\varepsilon>0$,
  \begin{itemize}
  \item [1).] if $1\leq p\leq 2\leq q\leq \infty$, $s>2$,
    \begin{align}\label{parabolic_conver_smoothing_coro_es_1}
      \begin{split}
              &\quad\|[g\cdot S_\varepsilon(f)-S_\varepsilon(g\cdot f)]^\varepsilon\cdot\psi\|_{L^1(\Omega_T\setminus\Omega_T^{1, \varepsilon})}\leq C\varepsilon[g]_{\mathscr{H}(\Omega_T; \bm{L^\infty})}\\&\qquad\cdot\{\|f\|_{L^{2, p^*}(\Omega_T; \bm{L^{s, 2}})}\|\psi\|_{L^{2, p'}(\Omega_T)}+\|f\|_{L^{\frac{2q}{q-2}, 2}(\Omega_T; \bm{L^{s, 2}})}\|\psi\|_{L^{q, 2}(\Omega_T)}\},
      \end{split}
    \end{align}
where $[\,\cdot\,]_{\mathscr{H}(\Omega_T; \bm{L^\infty})}$ is given as \eqref{parabolic_pre_def_C_seminorm} and $C$ depends only on $d, q, s$;
\item [2).] if $1\leq p\leq 2\leq q\leq \infty$, $s>2$,
    \begin{align}\label{parabolic_conver_smoothing_coro_es_3}
          \begin{split}
              &\quad\|\{S_\varepsilon([g\cdot S_\varepsilon(h)-S_\varepsilon(g\cdot h)]\cdot f)\}^\varepsilon\cdot\psi\|_{L^1(\Omega_T\setminus\Omega_T^{1, \varepsilon})}\leq C\varepsilon[g]_{\mathscr{H}(\Omega_T; \bm{L^\infty})}\\&\qquad\cdot\|h\|_{L^\infty(\Omega_T; \bm{L^{s, 2}})}\cdot \{\|f\|_{L^{2, p^*}(\Omega_T; \bm{L^\infty})}\|\psi\|_{L^{2, p'}(\Omega_T)}+\|f\|_{L^{\frac{2q}{q-2}, 2}(\Omega_T; \bm{L^{\infty}})}\|\psi\|_{L^{q, 2}(\Omega_T)}\},
      \end{split}
    \end{align}
where $C$ depends only on $d, q, s$;
\item [3).] for $p_0=\frac{2d}{d-1}$, $q_0=\frac{2d}{d+1}$, if $s>1$, $r\geq 1$,
  \begin{align}\label{parabolic_conver_smoothing_coro_es_4}
    \begin{split}
      &\quad\|[\{g\cdot S_\varepsilon(f)-S_\varepsilon(g\cdot f)\}\cdot S_\varepsilon(h)]^\varepsilon\|_{L^1(\Omega_T\setminus\Omega_T^{1, \varepsilon})}\leq C\varepsilon[g]_{\mathscr{H}(\Omega_T; \bm{L^\infty})}\\&\cdot\{\|f\|_{L^{2, q_0^*}(\Omega_T; \bm{L^{s_2, r_2}})}\|h\|_{L^{2, p_0}(\Omega_T; \bm{L^{s_3, r_3}})}+\|f\|_{L^{4, 2}(\Omega_T; \bm{L^{s_2, r_2}})}\|h\|_{L^{4, 2}(\Omega_T; \bm{L^{s_3, r_3}})}\},
    \end{split}
  \end{align}
  where $C$ depends only on $d, s$.
  \end{itemize}
\end{corollary}
\begin{proof}
  To show \eqref{parabolic_conver_smoothing_coro_es_1}, note that
  \begin{align*}
    &\quad\|[g\cdot S_\varepsilon(f)-S_\varepsilon(g\cdot f)]^\varepsilon\cdot\psi\|_{L^1}\\&\leq\|[g\cdot S_\varepsilon(f)-S^t_\varepsilon(g\cdot S_\varepsilon^x(f))]^\varepsilon\cdot\psi\|_{L^1}+\|[S^t_\varepsilon(g\cdot S^x_\varepsilon(f)-S^x_\varepsilon(g\cdot f))]^\varepsilon\cdot\psi\|_{L^1}\\&\leq \|[g\cdot S_\varepsilon(f)-S^t_\varepsilon(g\cdot S_\varepsilon^x(f))]^\varepsilon\|_{L^{q', 2}}\|\psi\|_{L^{q, 2}}+\|[S^t_\varepsilon(g\cdot S^x_\varepsilon(f)-S^x_\varepsilon(g\cdot f))]^\varepsilon\|_{L^{2, p}}\|\psi\|_{L^{2, p'}},
  \end{align*}
both of which can be handled in the same way as in the proof of Lemmas \ref{parabolic_pre_smooth_lem_3} and \ref{parabolic_pre_smooth_lem_4}. We omit the details. Estimates \eqref{parabolic_conver_smoothing_coro_es_3} and \eqref{parabolic_conver_smoothing_coro_es_4} can be derived similarly. 
\end{proof}

\begin{lemma}
  \label{parabolic_pre_smooth_lem_partt}
  Suppose $X$ is a Banach space and $g\in W^{\frac{1}{2}, r}(0, T; X)$, $1<r<\infty$. Then
  \begin{align*}
    \|\partial_t S^t_\varepsilon(g)\|_{L^r(\frac{\varepsilon^2}{2}, T-\frac{\varepsilon^2}{2}; X)}\leq C\varepsilon^{-1}[g]_{W^{\frac{1}{2}, r}(0, T; X)},
  \end{align*}
  where $S^t_\varepsilon(g)$ is defined in the sense of Bochner integral and $C$ depends only on $r$.
\end{lemma}
\begin{proof}
  For each $t\in (\varepsilon^2/2, T-\varepsilon^2/2)$,
  \begin{align*}
    \partial_tS^t_\varepsilon(g)(t)=\varepsilon^{-2}\int_{\mathbb{R}}(\partial_t\varphi_1)_\varepsilon(t-\varsigma)\cdot g(\varsigma)d\varsigma=\varepsilon^{-2}\int_{\mathbb{R}}(\partial_t\varphi_1)_\varepsilon(t-\varsigma)\cdot\{g(\varsigma)-g(t)\}d\varsigma,
  \end{align*}
where $(\partial_t\varphi_1)_\varepsilon(\varsigma)=\varepsilon^{-2}\partial_t\varphi_1(\varsigma/\varepsilon^2)$ and we have used the fact that $\int_{\mathbb{R}}(\partial_t\varphi_1)_\varepsilon =0$. This, together with the fact $\mathrm{supp}(\partial_t\varphi_1)_\varepsilon\subset (-\varepsilon^2/2, \varepsilon^2/2)$, implies that for $t\in (\varepsilon^2/2, T-\varepsilon^2/2)$
  \begin{align*}
    \|\partial_tS^t_\varepsilon(g)(t)\|_X&\leq \varepsilon^{-2}\int_0^T|(\partial_t\varphi_1)_\varepsilon(t-\varsigma)|\cdot\|g(\varsigma)-g(t)\|_Xd\varsigma\\&\leq \varepsilon^{-2}\Big(\int_0^T|(\partial_t\varphi_1)_\varepsilon(t-\varsigma)|^{r'}|t-\varsigma|^{(\frac{1}{2}+\frac{1}{r})r'}d\varsigma\Big)^{\frac{1}{r'}}\cdot \Big(\int_0^T\frac{\|g(\varsigma)-g(t)\|_X^r}{|\varsigma-t|^{1+r/2}}d\varsigma\Big)^{\frac{1}{r}}\\&\leq C\varepsilon^{-1}\Big(\int_0^T\frac{\|g(\varsigma)-g(t)\|_X^r}{|\varsigma-t|^{1+r/2}}d\varsigma\Big)^{\frac{1}{r}}.
  \end{align*}
The desired result follows directly.
\end{proof}

\begin{lemma}\label{parabolic_conver_pre_lem_1}
  (i). Suppose $f(x, t)\in L^q(\mathbb{R}; L^p(\mathbb{R}^d))$, $1\leq p, q\leq\infty$. Then for any $p\leq p_1\leq\infty$,
  \begin{align}\label{parabolic_pre_es_smoothing}
    \|S^x_\varepsilon(f)\|_{L^q(\mathbb{R}; L^{p_1}(\mathbb{R}^d))}\leq C\varepsilon^{\frac{d}{p_1}-\frac{d}{p}}\|f\|_{L^q(\mathbb{R}; L^p(\mathbb{R}^d))},
  \end{align}
  where $C$ depends only on $d, p, p_1$, and for any $q\leq q_1\leq \infty$,
  \begin{align}\label{parabolic_pre_es_smoothing_t}
    \|S^t_\varepsilon(f)\|_{L^{q_1}(\mathbb{R}; L^{p}(\mathbb{R}^d))}\leq C\varepsilon^{\frac{2}{q_1}-\frac{2}{q}}\|f\|_{L^q(\mathbb{R}; L^p(\mathbb{R}^d))},
  \end{align}
where $C$ depends only on $d, q, q_1$.
  
(ii). Suppose $f\in L^2(\mathbb{R}; W^{2, q}(\mathbb{R}^d))$ and $\partial_t f\in L^2(\mathbb{R}; L^q(\mathbb{R}^d))$. If $1\leq p<\infty$ and $\max\{1, \frac{dp}{d+p}\}\leq q\leq p$, or $p=\infty$ and $d(=\frac{dp}{d+p})<q\leq \infty$, then
\begin{align}\label{parabolic_pre_es_smoothing_2}
  \|\nabla f-S_\varepsilon(\nabla f)\|_{L^2(\mathbb{R}; L^p(\mathbb{R}^d))}\leq C\varepsilon^{1+\frac{d}{p}-\frac{d}{q}}\{\|\nabla^2 f\|_{L^2(\mathbb{R}; L^q(\mathbb{R}^d))}+\|\partial_tf\|_{L^2(\mathbb{R}; L^q(\mathbb{R}^d))}\},
\end{align}
where $C$ depends only on $d, p, q$.
\end{lemma}

\begin{proof}
  Estimates \eqref{parabolic_pre_es_smoothing} and \eqref{parabolic_pre_es_smoothing_t} are direct results of Young's inequality. We emphasize that \eqref{parabolic_pre_es_smoothing} is valid for any $\varphi_2\in C_0^\infty(B(0, 1/2))$, even if assumption \eqref{parabolic_pre_smooth_cond_unit} is not satisfied. To show \eqref{parabolic_pre_es_smoothing_2}, we write
  \begin{align*}
    \|\nabla f-S_\varepsilon(\nabla f)\|_{L^2(\mathbb{R}; L^p(\mathbb{R}^d))}\leq \|\nabla f-S^x_\varepsilon(\nabla f)\|_{L^2(\mathbb{R}; L^p(\mathbb{R}^d))}+\|S^x_\varepsilon(\nabla f)-S_\varepsilon(\nabla f)\|_{L^2(\mathbb{R}; L^p(\mathbb{R}^d))}.
  \end{align*}
  For the first term, by rescaling and Young's inequality, one can show that (\cite{Xu2019_stratified2})
  \begin{align}
    \label{parabolic_es_F_first}
    \|\nabla f-S^x_\varepsilon(\nabla f)\|_{L^2(\mathbb{R}; L^p(\mathbb{R}^d))}\leq C\varepsilon^{1+\frac{d}{p}-\frac{d}{q}}\|\nabla^2 f\|_{L^2(\mathbb{R}; L^q(\mathbb{R}^d))}.
  \end{align}
  On the other hand, it follows from inequality \eqref{parabolic_ineq_poincare} that
  \begin{align*}%\label{parabolic_conver_es_f-Sf}
    \begin{split}
          |S^x_\varepsilon(\nabla f)(t)-S_{\varepsilon}(\nabla f)(t)|&\leq C\fint_{\{|\varsigma-t|\leq \frac{\varepsilon^2}{2}\}}|S^x_\varepsilon(\nabla f)(t)-S^x_\varepsilon(\nabla f)(\varsigma)|d\varsigma\\&\leq C\varepsilon^2\fint_{\{|\varsigma-t|\leq \frac{\varepsilon^2}{2}\}}|\partial_t S^x_\varepsilon(\nabla f)(\varsigma)|d\varsigma,
    \end{split}
  \end{align*}
  which, together with \eqref{parabolic_pre_es_smoothing}, yields
  \begin{align*}
    \|S^x_\varepsilon(\nabla f)-S_\varepsilon(\nabla f)\|_{L^2(\mathbb{R}; L^p(\mathbb{R}^d))}&\leq C\varepsilon\Big\|\fint_{\{|\varsigma-t|\leq \frac{\varepsilon^2}{2}\}}| (\nabla \varphi_2)_\varepsilon*(\partial_tf)(\varsigma)|d\varsigma\Big\|_{L^2(\mathbb{R}; L^p(\mathbb{R}^d))}\\&\leq C\varepsilon^{1+\frac{d}{p}-\frac{d}{q}}\|\partial_tf\|_{L^2(\mathbb{R}; L^q(\mathbb{R}^d))}.
  \end{align*}
This completes the proof.
\end{proof}

In particular, for $q_0= \frac{2d}{d+1}$, we have the following estimates
  \begin{align}\label{parabolic_es_Sf}
    \begin{split}
      \|S^x_\varepsilon(f)\|_{L^2(\mathbb{R}^{d+1})}&\leq C\varepsilon^{-1/2}\|f\|_{L^2(\mathbb{R}; L^{q_0}(\mathbb{R}^d))}, \\
    \|\nabla f-S_\varepsilon(\nabla f)\|_{L^2(\mathbb{R}^{d+1})}&\leq C\varepsilon^{1/2}\{\|\nabla^2 f\|_{L^2(\mathbb{R}; L^{q_0}(\mathbb{R}^d))}+\|\partial_tf\|_{L^2(\mathbb{R}; L^{q_0}(\mathbb{R}^d))}\}.
    \end{split}
  \end{align}

\begin{lemma}\label{parabolic_conver_lem_u0-Su0}
  Let $g\in L^\infty(\Omega_T; \bm{L^2})$, $\partial_t f\in L^2(0, T; L^{q_0}(\Omega))$, $\nabla f\in L^2(0, T; \dot{W}^{1, q_0}(\Omega))$, and $h\in L^2(0, T; L^{p_0}(\Omega))$. Suppose further $\mathrm{supp}(h) \subset\Omega_T\setminus\Omega_T^{2, \varepsilon}$. Then for $\varepsilon>0$,
  \begin{align*}&\quad\Big|\iint_{\Omega_T}(\nabla f-S_\varepsilon(\nabla f))\cdot [S_\varepsilon(g\cdot h)]^\varepsilon dxdt\Big|\\&\leq C\varepsilon\{\|\nabla^2f\|_{L^{2, q_0}(\Omega_T)}+\|\partial_tf\|_{L^{2, q_0}(\Omega_T)}\}\cdot\|g\|_{L^\infty(\Omega_T; \bm{L^2})}\cdot\|h\|_{L^{2, p_0}(\Omega_T)},\end{align*}
  where $C$ depends only on $d$.
\end{lemma}
\begin{proof}
  We apply a regularity lifting argument of parabolic type to $g$. Write $g=g_1+g_2$, where
  \begin{align*}
    g_2(x, t)=\int_{\mathbb{T}^{d+1}}g(x, t; y, \tau)dyd\tau,
  \end{align*}
such that, $\int_{\mathbb{T}^{d+1}}g_1(x, t; \cdot, \cdot)=0$ for each $(x, t)$. Let $u(x, t; y, \tau)$ be the solution to
  \begin{equation*}
    \begin{cases}
      \partial_\tau u-\Delta_y u=g_1\quad \mathrm{in}~\mathbb{R}^{d+1},\\
     u \textrm{ is 1-periodic in } (y, \tau),\textrm{ and }\int_{\mathbb{T}^{d+1}} u(x, t; \cdot, \cdot)=0,
    \end{cases}
  \end{equation*}
which satisfies that, for each $(x, t)$
\begin{align*}
 \|u(x, t; \cdot, \cdot)\|_{L^2(\mathbb{T}^1; H^2(\mathbb{T}^d))\cap H^1(\mathbb{T}^1; L^2(\mathbb{T}^d))}\leq C\|g_1(x, t; \cdot, \cdot)\|_{L^2(\mathbb{T}^{d+1})} \leq C\|g(x, t; \cdot, \cdot)\|_{L^2(\mathbb{T}^{d+1})}.
\end{align*}
Set $G_k=-\partial_{y_k}u$ for $1\leq k\leq d$, and $G_{d+1}=u$. Then $\sum_{1\leq k\leq d}\partial_{y_k} G_k+\partial_\tau G_{d+1}=g_1$ and for each $(x, t)$
\begin{align*}
    \sum_{1\leq k\leq d}\|G_k\|_{L^2(\mathbb{T}^1; H^1(\mathbb{T}^d))\cap \bm{L^{\infty, 2}}}+\|G_{d+1}\|_{L^2(\mathbb{T}^1; H^2(\mathbb{T}^d))\cap H^1(\mathbb{T}^1; L^2(\mathbb{T}^d))}\leq C\|g\|_{L^2(\mathbb{T}^{d+1})},
\end{align*}
  which, by interpolation, yields that
  \begin{align}\label{parabolic_es_interpolation}
    \sum_{1\leq k\leq d}\|G_k\|_{L^\infty(\Omega_T; \bm{L^{4, p_0}})}+\|\nabla_yG_{d+1}\|_{L^\infty(\Omega_T; \bm{L^{4, p_0}})}\leq C\|g\|_{L^\infty(\Omega_T; \bm{L^2})},
  \end{align}
where $2\leq p_0\leq 4$ as $d\geq 2$. Therefore, with this expression of $g$, we have
\begin{align*}
    &\quad\Big|\iint_{\Omega_T}(\nabla f-S_\varepsilon(\nabla f))\cdot [S_\varepsilon(g\cdot h)]^\varepsilon dxdt\Big|\leq\Big|\iint_{\Omega_T}(\nabla f-S_\varepsilon(\nabla f))\cdot [S_\varepsilon(\partial_{y_k}G_k\cdot h)]^\varepsilon dxdt\Big|\\&
    +\Big|\iint_{\Omega_T}(\nabla f-S_\varepsilon(\nabla f))\cdot [S_\varepsilon(\partial_\tau G_{d+1}\cdot h)]^\varepsilon dxdt\Big|+\Big|\iint_{\Omega_T}(\nabla f-S_\varepsilon(\nabla f))\cdot [S_\varepsilon(g_2\cdot h)]^\varepsilon dxdt\Big|\\&\doteq K_1+K_2+K_3.
\end{align*}
Here we have omitted the summation w.r.t. $k$ from $1$ to $d$.

For $K_1$, it follows from \eqref{parabolic_intro_iden_composition} and integration by parts that
\begin{align*}
  K_1&=\varepsilon\Big|\iint_{\Omega_T}(\nabla f-S_\varepsilon(\nabla f))\cdot \{\partial_k[S_\varepsilon(G_k\cdot h)]^\varepsilon-[\partial_{x_k}S_\varepsilon(G_k\cdot h)]^\varepsilon\} dxdt\Big|\\&\leq \varepsilon\Big|\iint_{\Omega_T}\partial_k(\nabla f-S_\varepsilon(\nabla f))\cdot [S_\varepsilon(G_k\cdot h)]^\varepsilon dxdt\Big|\\&\qquad+\varepsilon\Big|\iint_{\Omega_T}(\nabla f-S_\varepsilon(\nabla f))\cdot [\partial_{x_k}S_\varepsilon(G_k\cdot h)]^\varepsilon dxdt\Big|\\&\leq C\varepsilon\|\nabla^2f\|_{L^{2, q_0}(\Omega_T)}\|[S_\varepsilon(G_k\cdot h)]^\varepsilon\|_{L^{2, p_0}(\Omega_T)}\\&\qquad+\|\nabla f-S_\varepsilon(\nabla f)\|_{L^{2, q_0}(\Omega_T\setminus\Omega_T^{1, \varepsilon})}\cdot \|[(\partial_k\varphi_2)_\varepsilon *S^t_\varepsilon(G_k\cdot h)]^\varepsilon\|_{L^{2, p_0}(\Omega_T)}\\&\leq C\varepsilon\{\|\nabla^2f\|_{L^{2, q_0}(\Omega_T)}+\|\partial_tf\|_{L^{2, q_0}(\Omega_T)}\}\cdot\|G_k\cdot h\|_{L^{2, p_0}(\Omega_T; \bm{L^{p_0}})}\\&\leq C\varepsilon\{\|\nabla^2f\|_{L^{2, q_0}(\Omega_T)}+\|\partial_tf\|_{L^{2, q_0}(\Omega_T)}\}\cdot\|g\|_{L^\infty(\Omega_T; \bm{L^2})}\cdot\|h\|_{L^{2, p_0}(\Omega_T)},
\end{align*}
where we have applied Lemma \ref{parabolic_pre_smooth_lem_2} and \eqref{parabolic_pre_es_smoothing_2} in the fourth step as well as \eqref{parabolic_es_interpolation} in the last one. Similarly,
\begin{align*}
  K_2&=\varepsilon^2\Big|\iint_{\Omega_T}(\nabla f-S_\varepsilon(\nabla f))\cdot \{\partial_t[S_\varepsilon(G_{d+1}\cdot h)]^\varepsilon-[\partial_tS_\varepsilon(G_{d+1}\cdot h)]^\varepsilon\} dxdt\Big|\\&\leq \varepsilon^2\Big|\iint_{\Omega_T}\partial_t(f-S_\varepsilon(f))\cdot \nabla[S_\varepsilon(G_{d+1}\cdot h)]^\varepsilon dxdt\Big|\\&\qquad+\varepsilon^2\Big|\iint_{\Omega_T}(\nabla f-S_\varepsilon(\nabla f))\cdot [\partial_tS_\varepsilon(G_{d+1}\cdot h)]^\varepsilon dxdt\Big|\\&\leq \varepsilon\iint_{\Omega_T}|\partial_t(f-S_\varepsilon(f))|\cdot |[(\nabla_x\varphi_2)_\varepsilon*S^t_\varepsilon(G_{d+1}\cdot h)]^\varepsilon+[S_{\varepsilon}(\nabla_yG_{d+1}\cdot h)]^\varepsilon| dxdt\\&\qquad+\iint_{\Omega_T}|\nabla f-S_\varepsilon(\nabla f)|\cdot |[(\partial_t\varphi_1)_\varepsilon*S^x_\varepsilon(G_{d+1}\cdot h)]^\varepsilon| dxdt\\&\leq C\varepsilon\{\|\partial_tf\|_{L^{2, q_0}(\Omega_T)}+\|\nabla^2f\|_{L^{2, q_0}(\Omega_T)}\}\cdot\|g\|_{L^\infty(\Omega_T; \bm{L^2})}\cdot\|h\|_{L^{2, p_0}(\Omega_T)},
\end{align*}
where $(\partial_t\varphi_1)_\varepsilon(t)=\varepsilon^{-2}(\partial_t\varphi_1)(t/\varepsilon^2)$. Lastly, it is easy to see from H\"{o}lder's inequality and \eqref{parabolic_pre_es_smoothing_2}
\begin{align*}
  K_3&\leq C\varepsilon\{\|\nabla^2f\|_{L^{2, q_0}(\Omega_T)}+\|\partial_tf\|_{L^{2, q_0}(\Omega_T)}\}\cdot\|g_2\|_{L^\infty(\Omega_T)}\cdot\|h\|_{L^{2, p_0}(\Omega_T)}\\&\leq C\varepsilon\{\|\nabla^2f\|_{L^{2, q_0}(\Omega_T)}+\|\partial_tf\|_{L^{2, q_0}(\Omega_T)}\}\cdot\|g\|_{L^\infty(\Omega_T; \bm{L^2})}\cdot\|h\|_{L^{2, p_0}(\Omega_T)}.
\end{align*}
By combining the estimates of $K_1$-$K_3$, we obtain the desired result.
\end{proof}

\subsection{Embeddings and boundary layers}
\label{parabolic_sec_boundary-layers}
In this subsection, we establish some embedding results for functions on $\Omega_T$ and derive the estimates of boundary layers.

\begin{lemma}[\cite{Shen2018_book, Xu2019_stratified2}]\label{parabolic_conver_boundary_lem_spatial}
   Suppose $u\in L^2(0, T; \dot{W}^{1, p}(\Omega))$, $\frac{2d}{d+2}\leq p\leq q_0=\frac{2d}{d+1}$. Then
  \begin{align*}
    \|u\|_{L^2(\Omega_\delta\times (0, T))}\leq C \delta^{1+\frac{d}{2}-\frac{d}{p}}\|u\|_{L^2(0, T; \dot{W}^{1, p}(\Omega))},
  \end{align*}
  where $C$ depends only on the Lipschitz character of $\Omega$.
\end{lemma}

\begin{lemma}\label{parabolic_conver_boundary_lem_embedding}
  Let $\Omega$ be a bounded Lipschitz domain in $\mathbb{R}^d$, $d\geq 2$. Suppose $\nabla u\in L^2(0, T; \dot{W}^{1, p}(\Omega))$ and $\partial_tu\in L^2(0, T; L^p(\Omega))$, where $\frac{2d}{d+2}\leq p<2$, $p>1$. Then
  \begin{align*}
    \|\nabla u\|_{L^q(0, T; L^2(\Omega))}\leq C\{\|\nabla u\|_{L^2(0, T; \dot{W}^{1, p}(\Omega))}+\|\partial_t u\|_{L^2(0, T; L^p(\Omega))}\},
  \end{align*}
  where $\frac{2}{q}=\frac{d}{p}-\frac{d}{2}$ and $C$ depends only on $d, p$ and the Lipschitz character of $\Omega$.
\end{lemma}
\begin{proof}
  By setting $v=\nabla u$, it is sufficient to show
  \begin{align*}
    \|v\|_{L^q(0, T; L^2(\Omega))}\leq C\{\|v\|_{L^2(0, T; \dot{W}^{1, p}(\Omega))}+\|\partial_t v\|_{L^2(0, T; \dot{W}^{-1, p}(\Omega))}\}.
  \end{align*}
  Since $\partial\Omega$ is Lipschitz, there exists a linear extension operator $P$, such that, $P$ is bounded from $\dot{W}^{1, p}(\Omega)$ into $\dot{W}^{1, p}(\mathbb{R}^d)$ with compact support and, at the same time, bounded from $\dot{W}^{-1, p}(\Omega)$ into $\dot{W}^{-1, p}(\mathbb{R}^d)$. Denoting $P(v)$ still by $v$, we have
  \begin{align*}
    \|v\|_{L^2(0, T; \dot{W}^{1, p}(\mathbb{R}^d))}+\|\partial_tv\|_{L^2(0, T; \dot{W}^{-1, p}(\mathbb{R}^d))}\leq C\{\|v\|_{L^2(0, T; \dot{W}^{1, p}(\Omega))}+\|\partial_t v\|_{L^2(0, T; \dot{W}^{-1, p}(\Omega))}\}.
  \end{align*}
  Therefore, it is sufficient to prove the case where $\Omega=\mathbb{R}^d$.

  To do this, denote $F=\partial_tv-\Delta v\in L^2(0, T; \dot{W}^{-1, p}(\mathbb{R}^d))$ and write it into $F=\partial_{x_j}f_j$, where $f_j\in L^2(0, T; L^p(\mathbb{R}^d))$. Set $$v_1(t)=\int_0^t e^{(t-s)\Delta}\partial_{x_j}f_j(s)ds,$$ which solves
\begin{align*}
  \partial_tv_1-\Delta v_1=F, \quad v_1(0)=0.
\end{align*}
By the calculation of heat kernel and the maximal regularity, we have
\begin{align*}
  \|v_1\|_{L^q(0, T; L^2(\mathbb{R}^d))}+\|\partial_tv_1\|_{L^2(0, T; \dot{W}^{-1, p}(\mathbb{R}^d))}+\|v_1\|_{L^2(0, T; \dot{W}^{1, p}(\mathbb{R}^d))}\leq C\|F\|_{L^2(0, T; \dot{W}^{-1, p}(\mathbb{R}^d))}.
\end{align*}
For the remaining part $v_2=v-v_1$, it holds that $\partial_tv_2-\Delta v_2=0$ and
\begin{align*}
  \|\partial_tv_2\|_{L^2(0, T; \dot{W}^{-1, p}(\mathbb{R}^d))}+\|v_2\|_{L^2(0, T; \dot{W}^{1, p}(\mathbb{R}^d))}\leq C\{\|\partial_tv\|_{L^2(0, T; \dot{W}^{-1, p}(\mathbb{R}^d))}+\|v\|_{L^2(0, T; \dot{W}^{1, p}(\mathbb{R}^d))}\}. 
\end{align*}
By the trace method of interpolation, we know
\begin{align*}
  \|v_2\|_{C([0, T]; X)}\leq C\{\|\partial_tv\|_{L^2(0, T; \dot{W}^{-1, p}(\mathbb{R}^d))}+\|v\|_{L^2(0, T; \dot{W}^{1, p}(\mathbb{R}^d))}\},
\end{align*}
where $X=(\dot{W}^{1, p}(\mathbb{R}^d), \dot{W}^{-1, p}(\mathbb{R}^d))_{\frac{1}{2}, 2}=\dot{B}^0_{p, 2}(\mathbb{R}^d)$ and $\dot{B}^0_{p, 2}(\mathbb{R}^d)$ is the homogeneous Besov space. Denote $w_j=\dot{\Delta}_jv_2$, where $\dot{\Delta}_j$ is the standard homogeneous dyadic blocks of Littlewood-Paley decomposition. Then $v_2=\sum_{j\in\mathbb{Z}}w_j$, $\|v_2(0)\|_{\dot{B}^0_{p, 2}}=\|(\|w_j(0)\|_{L^p})_{j\in\mathbb{Z}}\|_{l^2}$ and
\begin{align*}
  \partial_tw_j-\Delta w_j=0,\quad w_j(0)\in L^p(\mathbb{R}^d), 
\end{align*}
which, by the calculation of heat kernel, gives that
\begin{align*}
  \|w_j\|_{L^q(0, T; L^2(\mathbb{R}^d))}\leq C\|w_j(0)\|_{L^p(\mathbb{R}^d)}.
\end{align*}
Noticing that $q\geq 2$, this implies that
\begin{align*}
  \|v_2\|_{L^q(0, T;L^2)}&=\|\hat{v_2}\|_{L^q(0, T;L^2)}=\|\sum\nolimits_j\hat{w_j}\|_{L^q(0, T;L^2)}\leq \|\Big(\sum\nolimits_j|\hat{w_j}|^2\Big)^{1/2}\|_{L^q(0, T;L^2)}\\&\leq \Big(\sum\nolimits_j\|\hat{w_j}\|_{L^q(0, T;L^2)}^2\Big)^{1/2}\leq C\|\|w_j(0)\|_{L^p}\|_{l^2}=C\|v_2(0)\|_{\dot{B}^0_{p, 2}}\\&\leq C\{\|\partial_tv\|_{L^2(0, T; \dot{W}^{-1, p}(\mathbb{R}^d))}+\|v\|_{L^2(0, T; \dot{W}^{1, p}(\mathbb{R}^d))}\},
\end{align*}
where we have used the fact that $\mathrm{supp}\hat{w_j}\cap \mathrm{supp}\hat{w_{j'}}=\emptyset$ if $|j-j'|\geq 2$. Combining the estimates of $v_1$ and $v_2$, we obtain the desired result for $v$.
\end{proof}

By H\"{o}lder's inequality, Lemma \ref{parabolic_conver_boundary_lem_embedding} implies immediately the following estimate for temporal boundary layers.
\begin{corollary}\label{parabolic_conver_boundary_coro_time}
  Under the assumptions of Lemma \ref{parabolic_conver_boundary_lem_embedding}, we have for any $s\in [0, T-\delta^2]$
  \begin{align*}
    \|\nabla u\|_{L^2(\Omega\times (s, s+\delta^2))}\leq C\delta^{1+\frac{d}{2}-\frac{d}{p}}\{\|\nabla u\|_{L^2(0, T; \dot{W}^{1, p}(\Omega))}+\|\partial_t u\|_{L^2(0, T; L^p(\Omega))}\},
  \end{align*}
where $C$ depends only on $d, p$ and the Lipschitz character of $\Omega$.
\end{corollary}

\section{Convergence rates}
\label{parabolic_sec_conv-rates}

This section is devoted to establishing the sharp convergence rate for problem \eqref{parabolic_intro_eq1}. We always assume that $\partial\Omega\in C^{1, 1}$ henceforward.

Suppose $\eta_\varepsilon:=\eta_{1, \varepsilon}\cdot\eta_{2, \varepsilon}$, where $\eta_{1, \varepsilon}$ and $\eta_{2, \varepsilon}$ are two smooth cut-off functions on $(0, T)$ and $\Omega$, respectively, such that, $0\leq \eta_{1, \varepsilon}, \eta_{2, \varepsilon}\leq 1$ and
\begin{align}\label{parabolic_def_cutoff}
  \begin{split}
  &\mathrm{supp}(\eta_{1, \varepsilon})\subset(4\varepsilon^2, T-4\varepsilon^2),\quad \eta_{1, \varepsilon}=1 \textrm{  in } (5\varepsilon^2, T-5\varepsilon^2),\quad |(\eta_{1, \varepsilon})'|\leq C/\varepsilon^2,\\
  &\eta_{2, \varepsilon}=0 \textrm{  on } \Omega_{4\varepsilon},\quad \eta_{2, \varepsilon}=1 \textrm{  on } \Omega\setminus\Omega_{5\varepsilon}, \quad |\nabla \eta_{2, \varepsilon}|\leq C/\varepsilon.
\end{split}
\end{align}

Let $u_\varepsilon$ and $u_0$ be the solutions to problems \eqref{parabolic_intro_eq1} and \eqref{parabolic_intro_eq_u0} respectively. For the sake of simplicity, we extend $u_0$ onto $\mathbb{R}^d\times(0, T)$, such that, $$\|\partial_t u_0\|_{L^2(0, T; L^{q_0}(\mathbb{R}^d))}+\|u_0\|_{L^2(0, T; W^{2, q_0}(\mathbb{R}^d))}\leq C\|\partial_t u_0\|_{L^2(0, T; L^{q_0}(\Omega))}+\|u_0\|_{L^2(0, T; W^{2, q_0}(\Omega))}.$$By equations \eqref{parabolic_intro_eq1} and \eqref{parabolic_intro_eq_u0}, we calculate that
\begin{align*} &\quad(\partial_t+\mathcal{L}_\varepsilon)(u_\varepsilon-u_0)=-\textrm{div}\{(\widehat{A}-A^\varepsilon)\nabla u_0\}\\&=-\textrm{div}\{\widehat{A}\nabla u_0-A^\varepsilon \nabla u_0-A^\varepsilon[S_\varepsilon(\nabla_y\widetilde{\chi} K_\varepsilon(\nabla u_0))]^\varepsilon\}-\mathrm{div}\{A^\varepsilon[S_\varepsilon(\nabla_y\widetilde{\chi} K_\varepsilon(\nabla u_0))]^\varepsilon\},
\end{align*}
where $K_\varepsilon(\cdot):=S_\varepsilon(\cdot)\eta_\varepsilon$ with $S_\varepsilon$ defined in Section \ref{parabolic_subsec_smoothing} and notation \eqref{parabolic_conver_def_tilde} was used in $\widetilde{\chi}$. Note that we have regarded $\chi$ as functions on $\mathbb{R}^{d+1}\times \mathbb{R}^{d+1}$ having value $0$ outside of $\Omega_T$, due to the cut-off effect of $K_\varepsilon$ (see \eqref{parabolic_def_cutoff}). Recalling that $\nabla_y$ is commutative with the smoothing operator $S_\varepsilon$, we deduce from \eqref{parabolic_intro_iden_composition} that
\begin{align*}
  -\mathrm{div}\{A^\varepsilon[S_\varepsilon(\nabla_y\widetilde{\chi} K_\varepsilon(\nabla u_0))]^\varepsilon\}&=-\mathrm{div}\{A^\varepsilon[\nabla_yS_\varepsilon(\widetilde{\chi} K_\varepsilon(\nabla u_0))]^\varepsilon\}\\&= \mathcal{L}_\varepsilon(\varepsilon [S_\varepsilon(\widetilde{\chi} K_\varepsilon(\nabla u_0))]^\varepsilon)+\mathrm{div}\{A^\varepsilon\varepsilon[\nabla_x S_\varepsilon(\widetilde{\chi} K_\varepsilon(\nabla u_0))]^\varepsilon\}.
\end{align*}
Thus, by setting $\mathbf{B}_\varepsilon=(\mathbf{B}_{\varepsilon, i}^\alpha)$, $1\leq i\leq d, 1\leq \alpha\leq m$, where \begin{align*}%\label{parabolic_conver_def_Bbf}
 \mathbf{B}^\alpha_{\varepsilon, i}(x, t; y, \tau):=A^{\alpha\beta}_{ij}\partial_{j} u^\beta_0+A^{\alpha\gamma}_{ik}S_\varepsilon(\partial_{y_k}\widetilde{\chi}^{\gamma\beta}_j K_\varepsilon(\partial_{j} u^\beta_0))-\widehat{A}^{\alpha\beta}_{ij}\partial_{j} u^\beta_0,\end{align*}
it follows that
\begin{align}\label{parabolic_conver_iden_Lw_1} \begin{split}
    &\quad(\partial_t+\mathcal{L}_\varepsilon)(u_\varepsilon-u_0-\varepsilon [S_\varepsilon(\widetilde{\chi} K_\varepsilon(\nabla u_0))]^\varepsilon)\\&=\mathrm{div}\{(\mathbf{B}_\varepsilon)^\varepsilon\}+\mathrm{div}\{A^\varepsilon\varepsilon[\nabla_x S_\varepsilon(\widetilde{\chi} K_\varepsilon(\nabla u_0))]^\varepsilon\}-\varepsilon\partial_t([S_\varepsilon(\widetilde{\chi} K_\varepsilon(\nabla u_0))]^\varepsilon)\\&=\partial_i\{(\mathbf{B}_{\varepsilon, i})^\varepsilon-[S_\varepsilon(\widetilde{B}_{ij}K_\varepsilon(\partial_j u_0))]^\varepsilon\}+\partial_i\{[S_\varepsilon(\widetilde{B}_{ij}K_\varepsilon(\partial_j u_0))]^\varepsilon\}\\&\quad+\mathrm{div}\{A^\varepsilon\varepsilon[\nabla_x S_\varepsilon(\widetilde{\chi} K_\varepsilon(\nabla u_0))]^\varepsilon\}-\varepsilon\partial_t([S_\varepsilon(\widetilde{\chi} K_\varepsilon(\nabla u_0))]^\varepsilon),
    \end{split}
\end{align}
where $(B_{ij}^{\alpha\beta})$ is defined by \eqref{parabolic_def_B}. We emphasize that
\begin{align}\label{parabolic_c_iden_Bdifference}
  \begin{split}
    &\mathbf{B}_{\varepsilon, i}-S_\varepsilon(\widetilde{B}_{ij} K_\varepsilon(\partial_j u_0))=A_{ij}\partial_{j} u_0-S_\varepsilon(\widetilde{A}_{ij}K_\varepsilon(\partial_j u_0))+A_{ik}S_\varepsilon(\partial_{y_k}\widetilde{\chi}_j K_\varepsilon(\partial_{j} u_0))\\&\qquad-S_\varepsilon[S_\varepsilon(A_{ik}\partial_{y_k}\chi_j) K_\varepsilon(\partial_j u_0)]-\widehat{A}_{ij}\partial_{j} u_0+S_\varepsilon(S_\varepsilon(\widehat{A}_{ij})K_\varepsilon(\partial_j u_0)),
  \end{split}
\end{align}
where the main difference between $\mathbf{B}_\varepsilon$ and $S_\varepsilon(\widetilde{B}K_\varepsilon(\nabla u_0))$ focuses on the smoothing acts of $S_\varepsilon$.

Noticing that $\chi_j=-B_{(d+1)j}$, by Lemma \ref{parabolic_pre_lem_fkB}, we write
\begin{align*}
  &\quad\partial_i\{[S_\varepsilon(\widetilde{B}_{ij}K_\varepsilon(\partial_j u_0))]^\varepsilon\}-\varepsilon\partial_t([S_\varepsilon(\widetilde{\chi} K_\varepsilon(\nabla u_0))]^\varepsilon)\\&=\partial_i\{[S_\varepsilon([\partial_{y_k}\widetilde{\mathfrak{B}}_{kij}+\partial_\tau\widetilde{\mathfrak{B}}_{(d+1)ij}]K_\varepsilon(\partial_{j} u_0))]^\varepsilon\}+\partial_t\{\varepsilon[S_\varepsilon(\partial_{y_k}\widetilde{\mathfrak{B}}_{k(d+1)j}K_\varepsilon(\partial_j u_0))]^\varepsilon\}\\&=\partial_i\{[\partial_{y_k}S_\varepsilon(\widetilde{\mathfrak{B}}_{kij}K_\varepsilon(\partial_{j} u_0))]^\varepsilon\}+\partial_i\{[\partial_\tau S_\varepsilon(\widetilde{\mathfrak{B}}_{(d+1)ij}K_\varepsilon(\partial_j u_0))]^\varepsilon\}\\&\quad+\partial_t\{\varepsilon[\partial_{y_k}S_\varepsilon(\widetilde{\mathfrak{B}}_{k(d+1)j}K_\varepsilon(\partial_j u_0))]^\varepsilon\},
\end{align*}
where we have also used the fact that $\mathfrak{B}_{(d+1)(d+1)j}=0$ in the first equality, as well as the commutativity between $S_\varepsilon$ and the partial derivatives w.r.t. $(y, \tau)$ in the second one. In view of \eqref{parabolic_intro_iden_composition}, together with the skew-symmetry of $\mathfrak{B}$ in Lemma \ref{parabolic_pre_lem_fkB}, this yields that
\begin{align}\label{parabolic_conver_iden_Lw_2}
  \begin{split}
  &\quad\partial_i\{[S_\varepsilon(\widetilde{B}_{ij}K_\varepsilon(\partial_j u_0))]^\varepsilon\}-\varepsilon\partial_t([S_\varepsilon(\widetilde{\chi} K_\varepsilon(\nabla u_0))]^\varepsilon)\\&=\partial_i\{\varepsilon\partial_{k}([S_\varepsilon(\widetilde{\mathfrak{B}}_{kij}K_\varepsilon(\partial_{j} u_0))]^\varepsilon)-\varepsilon[\partial_{x_k}S_\varepsilon(\widetilde{\mathfrak{B}}_{kij}K_\varepsilon(\partial_{j} u_0))]^\varepsilon\}\\&\quad+\partial_i\{\varepsilon^2\partial_t([S_\varepsilon(\widetilde{\mathfrak{B}}_{(d+1)ij}K_\varepsilon(\partial_j u_0))]^\varepsilon)-\varepsilon^2[\partial_t S_\varepsilon(\widetilde{\mathfrak{B}}_{(d+1)ij}K_\varepsilon(\partial_j u_0))]^\varepsilon\}\\&\quad+\partial_t\{\varepsilon^2\partial_{k}([S_\varepsilon(\widetilde{\mathfrak{B}}_{k(d+1)j}K_\varepsilon(\partial_j u_0))]^\varepsilon)-\varepsilon^2[\partial_{x_k}S_\varepsilon(\widetilde{\mathfrak{B}}_{k(d+1)j}K_\varepsilon(\partial_j u_0))]^\varepsilon\}\\&=-\varepsilon\partial_i\{[\partial_{x_k}S_\varepsilon(\widetilde{\mathfrak{B}}_{kij}K_\varepsilon(\partial_{j} u_0))]^\varepsilon\}-\varepsilon^2\partial_i\{[\partial_t S_\varepsilon(\widetilde{\mathfrak{B}}_{(d+1)ij}K_\varepsilon(\partial_j u_0))]^\varepsilon\}\\&\quad-\varepsilon^2\partial_t\{[\partial_{x_k}S_\varepsilon(\widetilde{\mathfrak{B}}_{k(d+1)j}K_\varepsilon(\partial_j u_0))]^\varepsilon\}.
  \end{split}
\end{align}
Therefore, by defining
\begin{align}
  \label{parabolic_conver_def_w}
    w_\varepsilon:=u_\varepsilon-u_0-\varepsilon [S_\varepsilon(\widetilde{\chi} K_\varepsilon(\nabla u_0))]^\varepsilon-\varepsilon^2[\partial_{x_k}S_\varepsilon(\widetilde{\mathfrak{B}}_{(d+1)kj}K_\varepsilon(\partial_j u_0))]^\varepsilon,
\end{align}
and combining \eqref{parabolic_conver_iden_Lw_1} and \eqref{parabolic_conver_iden_Lw_2}, we finally get
\begin{align}\label{parabolic_conver_iden_Lw_3}
  \begin{split}
    &(\partial_t+\mathcal{L}_\varepsilon)w_\varepsilon=\partial_i\{(\mathbf{B}_{\varepsilon, i})^\varepsilon-[S_\varepsilon(\widetilde{B}_{ij}K_\varepsilon(\partial_j u_0))]^\varepsilon\}+\mathrm{div}\{A^\varepsilon\varepsilon[\nabla_x S_\varepsilon(\widetilde{\chi} K_\varepsilon(\nabla u_0))]^\varepsilon\}\\&\qquad+\mathrm{div}\{A^\varepsilon\nabla(\varepsilon^2[\partial_{x_k}S_\varepsilon(\widetilde{\mathfrak{B}}_{(d+1)kj}K_\varepsilon(\partial_j u_0))]^\varepsilon)\}-\varepsilon\partial_i\{[\partial_{x_k}S_\varepsilon(\widetilde{\mathfrak{B}}_{kij}K_\varepsilon(\partial_{j} u_0))]^\varepsilon\}\\&\qquad\qquad-\varepsilon^2\partial_i\{[\partial_t S_\varepsilon(\widetilde{\mathfrak{B}}_{(d+1)ij}K_\varepsilon(\partial_j u_0))]^\varepsilon\}.
  \end{split}
\end{align}

\begin{lemma}\label{parabolic_conver_lem_Lw}
  Let $\Omega$ be a bounded Lipschitz domain in $\mathbb{R}^d$ and $0<\varepsilon\leq \sqrt{T}$. Assume that $A$ satisfies \eqref{parabolic_intro_cond_elliptic}--\eqref{parabolic_intro_cond_AC}. Let $w_\varepsilon$ be defined by \eqref{parabolic_conver_def_w} with $K_\varepsilon(\cdot)=S_\varepsilon(\cdot)\eta_\varepsilon$. Then for any $\psi\in L^2(0, T; C_0^1(\Omega))$ and $1\leq p\leq 2\leq q\leq \infty$,
  \begin{align}
    &\Big|\int_0^T\langle\partial_tw_\varepsilon, \psi\rangle_{H^{-1}(\Omega)\times H^1_0(\Omega)}+\iint_{\Omega_T}A^\varepsilon\nabla w_\varepsilon\cdot\nabla\psi \Big|\leq C\|\nabla u_0-S_\varepsilon(\nabla u_0)\|_{L^{2, p}(\Omega_T\setminus\Omega_T^{4, \varepsilon})}\|\nabla\psi\|_{L^{2, p'}(\Omega_T)}\nonumber\\&+C\|\nabla u_0\|_{L^2(\Omega_T^{6, \varepsilon})}\|\nabla\psi\|_{L^2(\Omega_T^{6, \varepsilon})}+C\varepsilon\big\{\|S_\varepsilon(\nabla u_0)\|_{L^{2, p^*}(\Omega_T\setminus\Omega_T^{4, \varepsilon})}+\|S_\varepsilon(\nabla^2 u_0)\|_{L^{2, p}(\Omega_T\setminus\Omega_T^{4, \varepsilon})}\nonumber\\&+\|(\nabla\varphi_2)_\varepsilon*\partial_t u_0\|_{L^{2, p}(\Omega_T\setminus\Omega_T^{4, \varepsilon})}\big\}\|\nabla \psi\|_{L^{2, p'}(\Omega_T)}+C\varepsilon\|S_\varepsilon(\nabla u_0)\|_{L^{\frac{2q}{q-2}, 2}(\Omega_T\setminus\Omega_T^{4, \varepsilon})}\|\nabla\psi\|_{L^{q, 2}(\Omega_T)},\label{parabolic_conver_es_w_psi}
\end{align}
where $(\nabla\varphi_2)_\varepsilon(x)=\varepsilon^{-d}\nabla\varphi_2(\frac{x}{\varepsilon})$ and $C$ depends only on $d$, $\mu$, $q$, $[A]_{\mathscr{H}(\Omega_T; \bm{L^\infty})}$.
\end{lemma}
\begin{proof}
  According to \eqref{parabolic_conver_iden_Lw_3}, we have
   \begin{align}
  &\quad\Big|\int_0^T\langle\partial_tw_\varepsilon, \psi\rangle_{H^{-1}(\Omega)\times H^1_0(\Omega)}+\iint_{\Omega_T}A^\varepsilon\nabla w_\varepsilon\cdot\nabla\psi \Big|\nonumber\\&\leq \Big|\iint_{\Omega_T}[\mathbf{B}_{\varepsilon, i}-S_\varepsilon(\widetilde{B}_{ij} K_\varepsilon(\partial_j u_0))]^\varepsilon\partial_i\psi\Big|+C\varepsilon\iint_{\Omega_T}|[\nabla_x S_\varepsilon(\widetilde{\chi} K_\varepsilon(\nabla u_0))]^\varepsilon||\nabla\psi|\nonumber\\&\quad+\varepsilon\iint_{\Omega_T}|[\partial_{x_k}S_\varepsilon(\widetilde{\mathfrak{B}}_{kij}K_\varepsilon(\partial_{j} u_0))]^\varepsilon\partial_i \psi|+C\varepsilon^2\iint_{\Omega_T}|\nabla([\partial_{x_k}S_\varepsilon(\widetilde{\mathfrak{B}}_{(d+1)kj}K_\varepsilon(\partial_j u_0))]^\varepsilon)||\nabla\psi|\nonumber\\&\quad+\varepsilon^2\iint_{\Omega_T}|[\partial_t S_\varepsilon(\widetilde{\mathfrak{B}}_{(d+1)ij}K_\varepsilon(\partial_j u_0))]^\varepsilon\partial_i\psi|\doteq I_1+I_2+I_3+I_4+I_5,\label{parabolic_conver_es_Lw}
\end{align}
where $C$ depends only on $\mu$.

Firstly, by \eqref{parabolic_c_iden_Bdifference}, we have
\begin{align*}
    &\quad\mathbf{B}_{\varepsilon}-S_\varepsilon(\widetilde{B} K_\varepsilon(\nabla u_0))=(A-\widehat{A})\nabla u_0-S_\varepsilon(S_\varepsilon(A-\widehat{A})K_\varepsilon(\nabla u_0))+AS_\varepsilon(\nabla_{y}\widetilde{\chi}\cdot K_\varepsilon(\nabla u_0))\nonumber\\&\qquad\qquad\qquad\qquad\qquad\qquad-S_\varepsilon[S_\varepsilon(A\nabla_{y}\chi)\cdot K_\varepsilon(\nabla u_0)]\nonumber\\&=(A-\widehat{A})[\nabla u_0-S_\varepsilon(K_\varepsilon(\nabla u_0))]+(A-\widehat{A}) S_\varepsilon(K_\varepsilon(\nabla u_0))-S_\varepsilon[S_\varepsilon(A-\widehat{A}) K_\varepsilon(\nabla u_0)]\nonumber\\&\qquad+AS_\varepsilon(\nabla_{y}\widetilde{\chi}\cdot K_\varepsilon(\nabla u_0))-S_\varepsilon[S_\varepsilon(A\nabla_{y}\chi)\cdot K_\varepsilon(\nabla u_0)],%\label{parabolic_c_iden_Bdifference_2}
\end{align*}
which, by the definition of $K_\varepsilon$, yields that
\begin{align*}
  I_1&\leq C\iint_{\Omega_T}|\nabla u_0-S_\varepsilon(S_\varepsilon(\nabla u_0)\eta_\varepsilon)||\nabla\psi|\\&\quad+\iint_{\Omega_T}|(A-\widehat{A}) S_\varepsilon(K_\varepsilon(\nabla u_0))-S_\varepsilon(S_\varepsilon(A-\widehat{A}) K_\varepsilon(\nabla u_0))||\nabla\psi|\\&\quad+\iint_{\Omega_T}|[AS_\varepsilon(\nabla_y\widetilde{\chi}\cdot K_\varepsilon(\nabla u_0))-S_\varepsilon(S_\varepsilon(A\nabla_y \chi)\cdot K_\varepsilon(\nabla u_0))]^\varepsilon||\nabla\psi|\\&\doteq I_{11}+I_{12}+I_{13}.
\end{align*}
It is not hard to see that $I_{11}$ can be bounded by
\begin{align*}
  C\|\nabla u_0-S_\varepsilon(\nabla u_0)\|_{L^{2, p}(\Omega_T\setminus\Omega_T^{4, \varepsilon})}\|\nabla\psi\|_{L^{2, p'}(\Omega_T)}+C\|\nabla u_0\|_{L^2(\Omega_T^{6, \varepsilon})}\|\nabla\psi\|_{L^2(\Omega_T^{6, \varepsilon})}.
\end{align*}
Then we turn to $I_{13}$, as $I_{12}$ can be handled in the same manner and has the same estimate. Precisely, $I_{13}$ can be dominated by
\begin{align*}
&\iint_{\Omega_T}|[AS_\varepsilon(\nabla_y\widetilde{\chi}\cdot K_\varepsilon(\nabla u_0))-S_\varepsilon(A\nabla_y \widetilde{\chi}\cdot K_\varepsilon(\nabla u_0))]^\varepsilon||\nabla\psi|\\&\quad+\iint_{\Omega_T}|\{S_\varepsilon([A\nabla_y\widetilde{\chi}-S_\varepsilon(A\nabla_y \chi)]\cdot K_\varepsilon(\nabla u_0))\}^\varepsilon||\nabla\psi|,
\end{align*}
which, by applying estimates \eqref{parabolic_conver_smoothing_coro_es_1} and \eqref{parabolic_conver_smoothing_coro_es_3} to these two terms respectively, yields that
\begin{align*}
  I_{13}&\leq C\varepsilon\{\|\nabla_y\widetilde{\chi}\cdot K_\varepsilon(\nabla u_0)\|_{L^{2, p^*}(\Omega_T; \bm{L^{\bar{q}}})}+\|K_\varepsilon(\nabla u_0)\|_{L^{2, p^*}(\Omega_T)}\} \|\nabla \psi\|_{L^{2, p'}(\Omega_T)}\\&\qquad+C\varepsilon\{\|\nabla_y\widetilde{\chi}\cdot K_\varepsilon(\nabla u_0)\|_{L^{\frac{2q}{q-2}, 2}(\Omega_T; \bm{L^{\bar{q}}})}+\|K_\varepsilon(\nabla u_0)\|_{L^{\frac{2q}{q-2}, 2}(\Omega_T)}\}\|\nabla\psi\|_{L^{q, 2}(\Omega_T)}\\&\leq C\varepsilon\{\|S_\varepsilon(\nabla u_0)\|_{L^{2, p^*}(\Omega_T\setminus\Omega_T^{4, \varepsilon})}\|\nabla\psi\|_{L^{2, p'}(\Omega_T)}+\|S_\varepsilon(\nabla u_0)\|_{L^{\frac{2q}{q-2}, 2}(\Omega_T\setminus\Omega_T^{4, \varepsilon})}\|\nabla\psi\|_{L^{q, 2}(\Omega_T)}\},
\end{align*}
where Lemma \ref{parabolic_pre_lem_chi} was also used and $C$ depends only on $d$, $\mu$, $q$, $[A]_{\mathscr{H}(\Omega_T; \bm{L^\infty})}$. Therefore, $I_1$ can be bounded by
\begin{align*}
  &\quad C\|\nabla u_0-S_\varepsilon(\nabla u_0)\|_{L^{2, p}(\Omega_T\setminus\Omega_T^{4, \varepsilon})}\|\nabla\psi\|_{L^{2, p'}(\Omega_T)}+C\|\nabla u_0\|_{L^2(\Omega_T^{6, \varepsilon})}\|\nabla\psi\|_{L^2(\Omega_T^{6, \varepsilon})}\\&+C\varepsilon\{\|S_\varepsilon(\nabla u_0)\|_{L^{2, p^*}(\Omega_T\setminus\Omega_T^{4, \varepsilon})}\|\nabla\psi\|_{L^{2, p'}(\Omega_T)}+\|S_\varepsilon(\nabla u_0)\|_{L^{\frac{2q}{q-2}, 2}(\Omega_T\setminus\Omega_T^{4, \varepsilon})}\|\nabla\psi\|_{L^{q, 2}(\Omega_T)}\}.
\end{align*}

For $I_2$, by the definitions of $S_\varepsilon, K_\varepsilon$, and Lemmas \ref{parabolic_pre_lem_chi}, \ref{parabolic_pre_smooth_lem_2}, we have
\begin{align*}
  I_2&\leq C\varepsilon\iint_{\Omega_T}|\{S_\varepsilon[\nabla_x\widetilde{\chi} S_\varepsilon(\nabla u_0)\eta_\varepsilon+ \widetilde{\chi} S_\varepsilon(\nabla^2 u_0)\eta_\varepsilon+\widetilde{\chi} S_\varepsilon(\nabla u_0)\nabla \eta_\varepsilon]\}^\varepsilon||\nabla\psi|\\&\leq C\varepsilon\{\|\nabla_x\widetilde{\chi}S_\varepsilon(\nabla u_0)\eta_\varepsilon\|_{L^{2, p}(\Omega_T; \bm{L^{2, p}})}+\|\widetilde{\chi}S_\varepsilon(\nabla^2 u_0)\eta_\varepsilon\|_{L^{2, p}(\Omega_T; \bm{L^{2, p}})}\}\cdot\|\nabla\psi\|_{L^{2, p'}(\Omega_T)}\\&\quad+C\|\nabla u_0\|_{L^2(\Omega_{6\varepsilon}\times(0, T))}\|\nabla\psi\|_{L^2(\Omega_{6\varepsilon}\times (0, T))}\\&\leq C\varepsilon\{\|S_\varepsilon(\nabla u_0)\|_{L^{2, p^*}(\Omega_T\setminus\Omega_T^{4, \varepsilon})}+\|S_\varepsilon(\nabla^2 u_0)\|_{L^{2, p}(\Omega_T\setminus\Omega_T^{4, \varepsilon})}\}\cdot\|\nabla\psi\|_{L^{2, p'}(\Omega_T)}\\&\quad+C\|\nabla u_0\|_{L^2(\Omega_{6\varepsilon}\times(0, T))}\|\nabla\psi\|_{L^2(\Omega_{6\varepsilon}\times (0, T))},
\end{align*}
where we have used the fact that $S_\varepsilon[\widetilde{\chi} S_\varepsilon(\nabla u_0)\nabla \eta_\varepsilon]=S_\varepsilon[\widetilde{\chi} S_\varepsilon(\nabla u_0)\nabla \eta_\varepsilon]\mathbf{1}_{\Omega_{6\varepsilon}}$, and $C$ depends only on $d$, $\mu$, $[A]_{\mathscr{H}(\Omega_T; \bm{L^\infty})}$.
Similarly, $I_3$ has the same bound as $I_2$.

To handle $I_4$, we write
\begin{align*}
  I_4&\leq C\varepsilon^2\iint_{\Omega_T}\{|[\nabla_x\partial_{x_k}S_\varepsilon(\widetilde{\mathfrak{B}}_{(d+1)kj}K_\varepsilon(\partial_j u_0))]^\varepsilon|+\varepsilon^{-1}|[\nabla_y\partial_{x_k}S_\varepsilon(\widetilde{\mathfrak{B}}_{(d+1)kj}K_\varepsilon(\partial_j u_0))]^\varepsilon|\}|\nabla\psi|\\&\leq C\varepsilon^2\iint_{\Omega_T}|[\nabla_x^2 S_\varepsilon(\widetilde{\mathfrak{B}}_{d+1}K_\varepsilon(\nabla u_0))]^\varepsilon||\nabla\psi|+C\varepsilon\iint_{\Omega_T}|[\nabla_xS_\varepsilon(\nabla_y\widetilde{\mathfrak{B}}_{d+1}K_\varepsilon(\nabla u_0))]^\varepsilon||\nabla\psi|,
\end{align*}
where we have used \eqref{parabolic_intro_iden_composition} and the notation $\mathfrak{B}_{d+1}=(\mathfrak{B}_{(d+1)kj})$. By using the arguments of $I_2$ together with Lemma \ref{parabolic_pre_lem_fkB} and Remark \ref{parabolic_pre_remark_1}, it follows that
\begin{align*}
    I_4&\leq C\varepsilon\iint_{\Omega_T}|\{(\nabla\varphi_2)_\varepsilon*S^t_\varepsilon(\nabla_x[\widetilde{\mathfrak{B}}_{d+1}K_\varepsilon(\nabla u_0)])\}^\varepsilon||\nabla\psi|+|\{S_\varepsilon(\nabla_x[\nabla_y\widetilde{\mathfrak{B}}_{d+1}K_\varepsilon(\nabla u_0)])\}^\varepsilon||\nabla\psi|\\&\leq C\varepsilon\{\|S_\varepsilon(\nabla u_0)\|_{L^{2, p^*}(\Omega_T\setminus\Omega_T^{4, \varepsilon})}+\|S_\varepsilon(\nabla^2 u_0)\|_{L^{2, p}(\Omega_T\setminus\Omega_T^{4, \varepsilon})}\}\cdot\|\nabla\psi\|_{L^{2, p'}(\Omega_T)}\\&\quad+C\|\nabla u_0\|_{L^2(\Omega_{6\varepsilon}\times(0, T))}\|\nabla\psi\|_{L^2(\Omega_{6\varepsilon}\times (0, T))}.
\end{align*}

For the last term, we have
\begin{align*}
I_5&\leq \varepsilon^2\iint_{\Omega_T}|[S_\varepsilon(\partial_t\widetilde{\mathfrak{B}}_{d+1}K_\varepsilon(\nabla u_0))+S_\varepsilon(\widetilde{\mathfrak{B}}_{d+1}\partial_tK_\varepsilon(\nabla u_0))]^\varepsilon||\nabla\psi|\\&\leq \varepsilon^2\iint_{\Omega_T}|[S_\varepsilon(\partial_t\widetilde{\mathfrak{B}}_{d+1}K_\varepsilon(\nabla u_0))]^\varepsilon||\nabla\psi|+\varepsilon^2\iint_{\Omega_T}|[S_\varepsilon(\widetilde{\mathfrak{B}}_{d+1}\nabla S_\varepsilon(\partial_tu_0)\eta_\varepsilon)]^\varepsilon||\nabla\psi|\\&\quad+\varepsilon^2\iint_{\Omega_T}|[S_\varepsilon(\widetilde{\mathfrak{B}}_{d+1}S_\varepsilon(\nabla u_0)\partial_t\eta_\varepsilon)]^\varepsilon||\nabla\psi|\doteq I_{51}+I_{52}+I_{53},
\end{align*}
where we have used the fact
\begin{align*}
  \partial_tK_\varepsilon(\nabla u_0)=S_\varepsilon(\nabla\partial_tu_0)\eta_\varepsilon+S_\varepsilon(\nabla u_0)\partial_t\eta_\varepsilon=\nabla S_\varepsilon(\partial_tu_0)\eta_\varepsilon+S_\varepsilon(\nabla u_0)\partial_t\eta_\varepsilon.
\end{align*}
To bound $I_{51}$, we apply Lemmas \ref{parabolic_pre_lem_fkB}, \ref{parabolic_pre_smooth_lem_2} and \ref{parabolic_pre_smooth_lem_partt} to obtain
\begin{align*}
  I_{51}\leq C\varepsilon^2\|\partial_t\widetilde{\mathfrak{B}}_{d+1}K_\varepsilon(\nabla u_0)\|_{L^{q', 2}(\Omega_T; \bm{L^2})}\|\nabla\psi\|_{L^{q, 2}(\Omega_T)}\leq C\varepsilon\|S_\varepsilon(\nabla u_0)\|_{L^{\frac{2q}{q-2}, 2}(\Omega_T\setminus\Omega_T^{4, \varepsilon})}\|\nabla\psi\|_{L^{q, 2}(\Omega_T)},
\end{align*}
where $C$ depends only on $d$, $\mu$, $[A]_{\mathscr{H}(\Omega_T; \bm{L^\infty})}$. Moreover, it follows from Lemmas \ref{parabolic_pre_lem_fkB} and \ref{parabolic_pre_smooth_lem_2} that
\begin{align*}
  I_{52}&\leq C\varepsilon^2\|\widetilde{\mathfrak{B}}_{d+1}\nabla S_\varepsilon(\partial_t u_0)\eta_\varepsilon\|_{L^{2, p}(\Omega_T; \bm{L^{2, p}})}\|\nabla\psi\|_{L^{2, p'}(\Omega_T)}\\&\leq C\varepsilon\|(\nabla\varphi_2)_\varepsilon*\partial_t u_0\|_{L^{2, p}(\Omega_T\setminus\Omega_T^{4, \varepsilon})}\|\nabla\psi\|_{L^{2, p'}(\Omega_T)},\end{align*}
where $(\nabla\varphi_2)_\varepsilon(x)=\varepsilon^{-d}\nabla\varphi_2(\frac{x}{\varepsilon})$ and $C$ depends only on $d$, $\mu$. On the other hand,
\begin{align*}I_{53}&= \varepsilon^2\iint_{\Omega_T}|[S_\varepsilon(\widetilde{\mathfrak{B}}_{d+1}S_\varepsilon(\nabla u_0)\partial_t\eta_\varepsilon \cdot \mathbf{1}_{(0, 5\varepsilon^2)\cup(T-5\varepsilon^2, T)})]^\varepsilon||\nabla\psi|\\&\leq C\|\nabla u_0\|_{L^2(\Omega\times [(0, 6\varepsilon^2)\cup(T-6\varepsilon^2, T)])}\|\nabla\psi\|_{L^2(\Omega\times [(0, 6\varepsilon^2)\cup(T-6\varepsilon^2, T)])},
\end{align*}
where we have applied Lemma \ref{parabolic_pre_smooth_lem_2} in the second step. As a result, $I_5$ can be bounded by
\begin{align*}
  &C\varepsilon\|S_\varepsilon(\nabla u_0)\|_{L^{\frac{2q}{q-2}, 2}(\Omega_T\setminus\Omega_T^{4, \varepsilon})}\|\nabla\psi\|_{L^{q, 2}(\Omega_T)}+C\varepsilon\|(\nabla\varphi_2)_\varepsilon*\partial_t u_0\|_{L^{2, p}(\Omega_T\setminus\Omega_T^{4, \varepsilon})}\|\nabla\psi\|_{L^{2, p'}(\Omega_T)}\\&+C\|\nabla u_0\|_{L^2(\Omega\times [(0, 6\varepsilon^2)\cup(T-6\varepsilon^2, T)])}\|\nabla\psi\|_{L^2(\Omega\times [(0, 6\varepsilon^2)\cup(T-6\varepsilon^2, T)])}.
\end{align*}

By combining \eqref{parabolic_conver_es_Lw} and the estimates of $I_1$--$I_5$, we conclude the desired estimate.
\end{proof}

\begin{theorem}\label{parabolic_conver_thm_H1_1}
  Suppose the assumptions of Lemma \ref{parabolic_conver_lem_Lw} hold. Then for $\frac{2d}{d+2}\leq q\leq q_0=\frac{2d}{d+1}$, $$\|w_\varepsilon\|_{L^2(0, T; H^1(\Omega))}\leq C\varepsilon^{1+\frac{d}{2}-\frac{d}{q}}\{\|\nabla u_0\|_{L^2(0, T; \dot{W}^{1, q}(\Omega))}+\|\partial_t u_0\|_{L^2(0, T; L^q(\Omega))}\},$$
  where $C$ depends only on $d, \mu, q, [A]_{\mathscr{H}(\Omega_T; \bm{L^\infty})}$ and the Lipschitz character of $\Omega$. In particular,
  $$\|w_\varepsilon\|_{L^2(0, T; H^1(\Omega))}\leq C\varepsilon^{1/2}\{\|\nabla u_0\|_{L^2(0, T; \dot{W}^{1, q_0}(\Omega))}+\|\partial_t u_0\|_{L^2(0, T; L^{q_0}(\Omega))}\}.$$
\end{theorem}
\begin{proof}
  Note that $w_\varepsilon(t)\in H^1_0(\Omega)$ for all $t\in [0, T]$ and $w_\varepsilon(0)=0$. By setting $p=q=2$ and $\psi=w_\varepsilon$ in \eqref{parabolic_conver_es_w_psi}, we have
  \begin{align*}%\label{parabolic_conver_es_I1-5}
    \begin{split}
      &\|w_\varepsilon\|_{L^2(0, T; H^1(\Omega))}\leq C\|\nabla u_0-S_\varepsilon(\nabla u_0)\|_{L^{2}(\Omega_T\setminus\Omega_T^{4, \varepsilon})}+C\|\nabla u_0\|_{L^2(\Omega_T^{6, \varepsilon})}+C\varepsilon\|S_\varepsilon(\nabla u_0)\|_{L^{2, 2^*}(\Omega_T\setminus\Omega_T^{4, \varepsilon})}\\&+C\varepsilon\|S_\varepsilon(\nabla^2 u_0)\|_{L^{2}(\Omega_T\setminus\Omega_T^{4, \varepsilon})}+C\varepsilon\|(\nabla\varphi_2)_\varepsilon*\partial_t u_0\|_{L^{2}(\Omega_T\setminus\Omega_T^{4, \varepsilon})}+C\varepsilon\|S_\varepsilon(\nabla u_0)\|_{L^{\infty, 2}(\Omega_T\setminus\Omega_T^{4, \varepsilon})}\\&\doteq J_1+J_2+J_3+J_4+J_5+J_6.
    \end{split}
  \end{align*}
  By Lemma \ref{parabolic_conver_pre_lem_1}, it is not hard to see that
  \begin{align*}
    J_1+J_3+J_4+J_5\leq C\varepsilon^{1+\frac{d}{2}-\frac{d}{q}}\{\|\nabla u_0\|_{L^2(0, T; \dot{W}^{1, q}(\Omega))}+\|\partial_t u_0\|_{L^2(0, T; L^q(\Omega))}\}.
  \end{align*}
Similarly, it follows from Lemmas \ref{parabolic_conver_pre_lem_1} and \ref{parabolic_conver_boundary_lem_embedding} that for $\frac{2}{r}=\frac{d}{q}-\frac{d}{2}$
  \begin{align*}
    J_6\leq C\varepsilon^{1-\frac{2}{r}}\|\nabla u_0\|_{L^r(0, T; L^2(\Omega))}\leq C\varepsilon^{1+\frac{d}{2}-\frac{d}{q}}\{\|\nabla u_0\|_{L^2(0, T; \dot{W}^{1, q}(\Omega))}+\|\partial_t u_0\|_{L^2(0, T; L^q(\Omega))}\}.
  \end{align*}
  On the other hand, thanks to Lemma \ref{parabolic_conver_boundary_lem_spatial} and Corollary \ref{parabolic_conver_boundary_coro_time}, we have
  \begin{align}\label{parabolic_conver_es_I2_boundary}
    \begin{split}
      J_2&\leq C\{\|\nabla u_0\|_{L^2(\Omega_{6\varepsilon}\times (0, T))}+\|\nabla u_0\|_{L^2(\Omega\times [(0, 6\varepsilon^2)\cup(T-6\varepsilon^2, T)])}\}\\&\leq C\varepsilon^{1+\frac{d}{2}-\frac{d}{q}}\{\|\nabla u_0\|_{L^2(0, T; \dot{W}^{1, q}(\Omega))}+\|\partial_t u_0\|_{L^2(0, T; L^q(\Omega))}\}.
    \end{split}
  \end{align}
  This completes the proof.
\end{proof}

\begin{corollary}\label{parabolic_conver_coro_H1}
  Suppose the assumptions of Lemma \ref{parabolic_conver_lem_Lw} hold. Then for any $\psi\in L^2(0, T; C_0^1(\Omega))$,
  \begin{align*}
    \begin{split}
    &\quad\Big|\int_0^T\langle\partial_tw_\varepsilon, \psi\rangle_{H^{-1}(\Omega)\times H^1_0(\Omega)}+\iint_{\Omega_T}A^\varepsilon\nabla w_\varepsilon\cdot\nabla\psi \Big|\\&\leq C\varepsilon^{1/2}\{\|\nabla u_0\|_{L^2(0, T; \dot{W}^{1, q_0}(\Omega))}+\|\partial_t u_0\|_{L^2(0, T; L^{q_0}(\Omega))}\}\cdot\|\nabla\psi\|_{L^2(\Omega_T)}.
  \end{split}
  \end{align*}
\end{corollary}

Now we prove the optimal $O(\varepsilon)$-convergence rate in $L^2(0, T; L^{p_0}(\Omega))$ stated in Theorem \ref{parabolic_conver_thm_conver}. We first introduce the dual problem. For $F\in C_0^\infty(\Omega_T)$, let $v_\varepsilon$ ($\varepsilon\geq 0$) be the weak solution to the following problem
\begin{equation}\label{parabolic_conver_eq_dual}
\begin{cases}
  -\partial_t v_\varepsilon+\mathcal{L}^*_\varepsilon v_\varepsilon =F  &\mathrm{in}~\Omega\times (0, T),\\
  v_\varepsilon=0 &\mathrm{on}~\partial \Omega\times(0, T),\\
 v_\varepsilon=0  & \mathrm{on}~ \Omega\times\{t=T\},
\end{cases}
\end{equation}
where $\mathcal{L}_\varepsilon^*$ is the adjoint operator of $\mathcal{L}_\varepsilon$ ($\varepsilon\geq 0$). By setting $\widehat{v}_\varepsilon(s)=v_\varepsilon(-s)$, one can see that $\widehat{v}_\varepsilon$ solves
\begin{equation}\label{parabolic_conver_eq_dual_hat}
\begin{cases}
  \partial_s \widehat{v}_\varepsilon+\widehat{\mathcal{L}}^*_\varepsilon \widehat{v}_\varepsilon =\widehat{F}  &\mathrm{in}~\Omega\times (-T, 0),\\
  \widehat{v}_\varepsilon=0 &\mathrm{on}~\partial \Omega\times(-T, 0),\\
 \widehat{v}_\varepsilon=0  & \mathrm{on}~ \Omega\times\{s=-T\},
\end{cases}
\end{equation}
where $\widehat{\mathcal{L}}^*_\varepsilon$ is the operator given by \eqref{parabolic_intro_exp_L} with $A(x, t; y, \tau)$ replaced by $A^*(x, -t; y, -\tau)$ and $\widehat{F}(x, s)=F(x, -s)$. Observe that $A^*(x, -t; y, -\tau)$ satisfies the same conditions on $\Omega\times (-T, 0)$ as $A$. Thus, the process and results discussed above for problem \eqref{parabolic_intro_eq1} remain valid for problem \eqref{parabolic_conver_eq_dual_hat}, thereby holding for problem \eqref{parabolic_conver_eq_dual}. Especially, the correctors $\chi^*$ and flux correctors $\mathfrak{B}^*$ could be introduced for the operator $-\partial_t+\mathcal{L}_\varepsilon^*$ (see also \eqref{parabolic_pre_eq_chi_dual}).

Define
\begin{align}
  \label{parabolic_conver_def_varpi}
  \varpi_\varepsilon=v_\varepsilon-v_0-\varepsilon [S_\varepsilon(\widetilde{\chi}^* K_\varepsilon(\nabla v_0))]^\varepsilon-\varepsilon^2[\partial_{x_k}S_\varepsilon(\widetilde{\mathfrak{B}}_{(d+1)kj}^* K_\varepsilon(\partial_j v_0))]^\varepsilon.
\end{align}
Then $\varpi_\varepsilon$ has the same estimates as $w_\varepsilon$. As like Theorem \ref{parabolic_conver_thm_H1_1}, we have,

\begin{corollary}\label{parabolic_conver_coro_varpi}
  Let $\varpi_\varepsilon$ be defined by \eqref{parabolic_conver_def_varpi}. Then
  \begin{align*}
    \|\varpi_\varepsilon\|_{L^2(0, T; H^1(\Omega))}\leq C\varepsilon^{1/2}\{\|\nabla v_0\|_{L^2(0, T; \dot{W}^{1, q_0}(\Omega))}+\|\partial_t v_0\|_{L^2(0, T; L^{q_0}(\Omega))}\},
  \end{align*}
where $C$ depends only on $d, \mu, [A]_{\mathscr{H}(\Omega_T; \bm{L^\infty})}$ and the Lipschitz character of $\Omega$.
\end{corollary}

\begin{lemma}\label{parabolic_conver_lem_v0}
  Suppose $\Omega$ is a bounded $C^{1, 1}$ domain in $\mathbb{R}^d$ and $A$ satisfies \eqref{parabolic_intro_cond_elliptic}--\eqref{parabolic_intro_cond_AC}. Let $v_0$ be the solution to problem \eqref{parabolic_conver_eq_dual} with $\varepsilon=0$. Then
  \begin{align}\label{parabolic_conver_es_v0_1}
    \|v_0\|_{L^2(0, T; W^{2, q_0}(\Omega))}+\|\partial_t v_0\|_{L^2(0, T; L^{q_0}(\Omega))}\leq C\|F\|_{L^2(0, T; L^{q_0}(\Omega))},
  \end{align}
where $C$ depends only on $d, \mu, \Omega, \|\nabla_xA\|_{L^{\infty, d}(\Omega_T; \bm{L^\infty})}$ and the $\mathrm{VMO}_x$ character $\varrho$ of $\widehat{A}$ given by \eqref{parabolic_conver_es_VMOx}. Consequently,
\begin{align}
  \label{parabolic_conver_es_v0_2}
  \|\nabla v_0\|_{L^2(\Omega^{6, \varepsilon}_T)}\leq C\varepsilon^{1/2}\|F\|_{L^2(0, T; L^{q_0}(\Omega))}.
\end{align}
\end{lemma}
\begin{proof}
 Due to Corollary \ref{parabolic_pre_coro_hatA}, $\widehat{A}\in W^{\frac{1}{2}, 2}(0, T; L^\infty(\Omega))\cap L^\infty(0, T; W^{1, d}(\Omega))$. We claim that $\widehat{A}$ satisfies the so-called $\mathrm{VMO}_x$ condition $$\lim_{R\rightarrow 0}\varrho(R)=0,$$ where
  \begin{align}\label{parabolic_conver_es_VMOx}
    \varrho(R):=\sup_{0<r<R}\sup_{\substack{B(x, r)\subset\Omega\\t\in(0, T-r^2)}}\fint_t^{t+r^2}\fint_{B(x, r)}\fint_{B(x, r)}|\widehat{A}(\omega, \varsigma)-\widehat{A}(z, \varsigma)|d\omega dzd\varsigma.
  \end{align}
Indeed, let $N\in \mathbb{N}^+$ and we decompose $[0, T)$ into $N$ intervals $E_k=[kT/N, (k+1)T/N)$, $k=0, \dots, N-1$. Set
\begin{align*}
  \mathfrak{A}_N(x, t)=\fint_{E_k}\widehat{A}(x, \varsigma)d\varsigma\quad \mathrm{if~}t\in E_k, \quad k=0, \dots, K-1.
\end{align*}
Then $\mathfrak{A}_N$ is finitely-valued w.r.t. $t$ and $\mathfrak{A}_N(\cdot, t)\in W^{1, d}(\Omega)$ for each $t$. By using inequality \eqref{parabolic_ineq_poincare}, together with H\"{o}lder's inequality,  we calculate that
\begin{align*}
  \fint_{B(x, r)}\fint_{B(x, r)}|\mathfrak{A}_N(\omega, t)-\mathfrak{A}_N(z, t)|d\omega dz&\leq C_d\fint_{B(x, r)}\int_{B(x ,r)}|\nabla \mathfrak{A}_N(y, t)||y-z|^{1-d}dydz\\&\leq C\|\nabla\mathfrak{A}_N(\cdot, t)\|_{L^d(B(x, r))},
\end{align*}
which yields that $\mathfrak{A}_N(\cdot, t)$ is a VMO function on $\Omega$ for each $t$. Since $\mathfrak{A}_N$ is finitely-valued, we have
\begin{align}\label{parabolic_conver_es_VMO}
    \lim_{R\rightarrow 0}\sup_{0<r<R}\sup_{\substack{B(x, r)\subset\Omega\\t\in(0, T)}}\fint_{B(x, r)}\fint_{B(x, r)}|\mathfrak{A}_N(\omega, t)-\mathfrak{A}_N(z, t)|d\omega dz=0.
\end{align}
Moreover, since $\widehat{A}\in W^{\frac{1}{2}, 2}(0, T; L^\infty(\Omega))$, it holds that
\begin{align}\label{parabolic_conver_es_VMO_2}
  \lim_{r\rightarrow 0}\sup_{t\in (0, T-r^2)}[\widehat{A}]_{W^{\frac{1}{2}, 2}(t, t+r^2; L^\infty(\Omega))}=0.
\end{align}
For $0<r<R$, $B(x, r)\subset\Omega$,
\begin{align*}
&\quad\sup_{t\in (0, T-r^2)}\fint_{t}^{t+r^2}\fint_{B(x, r)}\fint_{B(x, r)}|\widehat{A}(\omega, \varsigma)-\widehat{A}(z, \varsigma)|d\omega dzd\varsigma\\&\leq \sup_{t\in(0, T-r^2)}\fint_t^{t+r^2}\Big[2\fint_{B(x, r)}|\widehat{A}(\omega, \varsigma)-\mathfrak{A}_N(\omega, \varsigma)|d\omega+\fint_{B(x, r)}\fint_{B(x, r)}|\mathfrak{A}_N(\omega, \varsigma)-\mathfrak{A}_N(z, \varsigma)|d\omega dz\Big]d\varsigma\\&\leq 4\sup_{t\in(0, T-r^2)}\fint_{t}^{t+r^2}\fint_{t}^{t+r^2}\|\widehat{A}(\cdot, \varsigma_1)-\widehat{A}(\cdot, \varsigma_2)\|_{L^\infty(\Omega)}d\varsigma_1 d\varsigma_2\\&\qquad\qquad+\sup_{t\in(0, T)}\fint_{B(x, r)}\fint_{B(x, r)}|\mathfrak{A}_N(\omega, t)-\mathfrak{A}_N(z, t)|d\omega dz\\&\leq C\sup_{t\in (0, T-r^2)}[\widehat{A}]_{W^{\frac{1}{2}, 2}(t, t+r^2; L^\infty(\Omega))}+\sup_{t\in(0, T)}\fint_{B(x, r)}\fint_{B(x, r)}|\mathfrak{A}_N(\omega, t)-\mathfrak{A}_N(z, t)|d\omega dz,
  \end{align*}
which, together with \eqref{parabolic_conver_es_VMO}--\eqref{parabolic_conver_es_VMO_2}, yields that $\lim_{R\rightarrow 0}\varrho(R)=0$.

  Now, thanks to $L^q$-$L^p$ estimates of non-divergence type parabolic systems with $\mathrm{VMO}_x$ coefficients in $C^{1, 1}$ cylinders (see \cite{Dong2018_Ap} and references therein for the problems on the whole space and half space, from which one can deduce the estimates for bounded cylinders), we have
  \begin{align*}
      &\quad\|\partial_t v_0\|_{L^2(0, T; L^{q_0}(\Omega))}+\|v_0\|_{L^2(0, T; W^{2, q_0}(\Omega))}\leq C\{\|F\|_{L^2(0, T; L^{q_0}(\Omega))}+\|\nabla \widehat{A}\nabla v_0\|_{L^2(0, T; L^{q_0}(\Omega))}\}\\&\leq C\{\|F\|_{L^2(0, T; L^{q_0}(\Omega))}+\|\nabla \widehat{A}\|_{L^\infty(0, T; L^d(\Omega))}\|\nabla v_0\|_{L^2(0, T; L^{q_0^*}(\Omega))}\}\leq C\|F\|_{L^2(0, T; L^{q_0}(\Omega))},
\end{align*}
where we have also used $L^q$-$L^p$ estimates of divergence type parabolic systems with $\mathrm{VMO}_x$ coefficients \cite{Dong2011_divergence} in the last inequality, and $C$ depends only on $d, \mu, \Omega, \varrho, \|\nabla_xA\|_{L^{\infty, d}(\Omega_T; \bm{L^\infty})}$. This gives \eqref{parabolic_conver_es_v0_1}. \eqref{parabolic_conver_es_v0_2} now follows from the same argument as \eqref{parabolic_conver_es_I2_boundary}.
\end{proof}

Armed with the previous results, we are now prepared to prove Theorem \ref{parabolic_conver_thm_conver} by using the duality argument initiated in \cite{Suslina2013_Dirichlet}. See also \cite{Geng2017_Convergence}.

\begin{proof}[\textbf{Proof of Theorem \ref{parabolic_conver_thm_conver}}]
  By interpolation, we deduce from Lemmas \ref{parabolic_pre_lem_chi} and \ref{parabolic_pre_lem_fkB} that $$\chi, \mathfrak{B}\in L^\infty(\Omega_T; \bm{L^{4, p_0}}).$$ Note that $2\leq p_0\leq 4$, as $d\geq 2$. Then Lemma \ref{parabolic_pre_smooth_lem_2}, together with Remark \ref{parabolic_pre_remark_1}, implies that
  \begin{align*}
    &\quad\varepsilon\|[S_\varepsilon(\widetilde{\chi} K_\varepsilon(\nabla u_0))]^\varepsilon\|_{L^{2, p_0}(\Omega_T)}+\varepsilon^2\|[\partial_{x_k}S_\varepsilon(\widetilde{\mathfrak{B}}_{(d+1)kj}K_\varepsilon(\partial_j u_0))]^\varepsilon\|_{L^{2, p_0}(\Omega_T)}\\&\leq C\varepsilon\{\|\widetilde{\chi}\|_{L^\infty(\Omega_T; \bm{L^{p_0}})}+\|\widetilde{\mathfrak{B}}_{d+1}\|_{L^\infty(\Omega_T; \bm{L^{p_0}})}\}\|\nabla u_0\|_{L^{2, p_0}(\Omega_T)}\\&\leq C\varepsilon\|\nabla u_0\|_{L^2(0, T; \dot{W}^{1, q_0}(\Omega))}.
  \end{align*}
  Thus, it is sufficient to bound $\|w_\varepsilon\|_{L^{2, p_0}(\Omega_T)}$. To do this, let $v_\varepsilon$ be the solution of problem \eqref{parabolic_conver_eq_dual} and $\varpi_\varepsilon$ be given by \eqref{parabolic_conver_def_varpi}. Then we have
\begin{align}
      &\quad\iint_{\Omega_T} w_\varepsilon\cdot F dxdt=\int_0^T\langle\partial_tw_\varepsilon, v_\varepsilon\rangle+\iint_{\Omega_T}A^\varepsilon\nabla w_\varepsilon\cdot\nabla v_\varepsilon\nonumber\\&=\Big(\int_0^T\langle\partial_tw_\varepsilon, \varpi_\varepsilon\rangle+\iint_{\Omega_T}A^\varepsilon\nabla w_\varepsilon\cdot\nabla \varpi_\varepsilon\Big)+\Big(\int_0^T\langle\partial_tw_\varepsilon, v_0\rangle+\iint_{\Omega_T}A^\varepsilon\nabla w_\varepsilon\cdot\nabla v_0\Big)\nonumber\\&\quad+\Big(\int_0^T\langle\partial_tw_\varepsilon, \varpi_\varepsilon-v_\varepsilon-v_0\rangle+\iint_{\Omega_T}A^\varepsilon\nabla w_\varepsilon\cdot\nabla (\varpi_\varepsilon-v_\varepsilon-v_0)\Big)\doteq Q_1+Q_2+Q_3.\label{parabolic_conver_es_I}
\end{align}

For $Q_1$, by Corollaries \ref{parabolic_conver_coro_H1} and \ref{parabolic_conver_coro_varpi}, we have
\begin{align*}
  \begin{split}
 &\quad Q_1\leq C\varepsilon^{1/2}\{\|\nabla u_0\|_{L^2(0, T; \dot{W}^{1, q_0}(\Omega))}+\|\partial_t u_0\|_{L^2(0, T; L^{q_0}(\Omega))}\}\cdot\|\nabla\varpi\|_{L^2(\Omega_T)}\\&\leq C\varepsilon\{\|\nabla u_0\|_{L^2(0, T; \dot{W}^{1, q_0}(\Omega))}+\|\partial_t u_0\|_{L^2(0, T; L^{q_0}(\Omega))}\}\cdot\{\|\nabla v_0\|_{L^2(0, T; \dot{W}^{1, q_0}(\Omega))}+\|\partial_t v_0\|_{L^2(0, T; L^{q_0}(\Omega))}\}\\&\leq C\varepsilon\{\|\nabla u_0\|_{L^2(0, T; \dot{W}^{1, q_0}(\Omega))}+\|\partial_t u_0\|_{L^2(0, T; L^{q_0}(\Omega))}\}\cdot\|F\|_{L^2(0, T; L^{q_0}(\Omega))},
  \end{split}
\end{align*}
where we have used estimate \eqref{parabolic_conver_es_v0_1} in the last inequality, and $C$ depends only on $d$, $\mu$, $\Omega$, $[A]_{\mathscr{H}(\Omega_T; \bm{L^\infty})}$ and the $\mathrm{VMO}_x$ character of $\widehat{A}$ (which can be reduced to the character of $A$).

To bound $Q_2$, we use \eqref{parabolic_conver_es_w_psi} with $p=q_0, q=4$ to obtain
\begin{align*}
  Q_2&\leq C\{\varepsilon^{1/2}\|\nabla v_0\|_{L^2(\Omega^{6, \varepsilon}_T)}+\varepsilon\|\nabla v_0\|_{L^{2, p_0}(\Omega_T)}+\varepsilon\|\nabla \psi\|_{L^{4, 2}(\Omega_T)}\}\\&\qquad\cdot\{\|\nabla u_0\|_{L^2(0, T; \dot{W}^{1, {q_0}}(\Omega))}+\|\partial_t u_0\|_{L^2(0, T; L^{q_0}(\Omega))}\}\\ &\leq C\varepsilon\{\|\nabla u_0\|_{L^2(0, T; \dot{W}^{1, q_0}(\Omega))}+\|\partial_t u_0\|_{L^2(0, T; L^{q_0}(\Omega))}\}\cdot\|F\|_{L^2(0, T; L^{q_0}(\Omega))},
\end{align*}
where we have used estimates \eqref{parabolic_es_Sf}, \eqref{parabolic_conver_es_I2_boundary} and Lemma \ref{parabolic_conver_boundary_lem_embedding} in the first inequality, as well as \eqref{parabolic_conver_es_v0_1}--\eqref{parabolic_conver_es_v0_2} in the last one.

To deal with $Q_3$, we start from \eqref{parabolic_conver_es_Lw} with
$$\psi=\varepsilon [S_\varepsilon(\widetilde{\chi}^* K_\varepsilon(\nabla v_0))]^\varepsilon+\varepsilon^2[\partial_{x_k}S_\varepsilon(\widetilde{\mathfrak{B}}_{(d+1)kj}^* K_\varepsilon(\partial_j v_0))]^\varepsilon.$$
Note that the latter term in $\psi$ can be written into $$\varepsilon^2[\partial_{x_k}S_\varepsilon(\widetilde{\mathfrak{B}}_{(d+1)kj}^* K_\varepsilon(\partial_j v_0))]^\varepsilon=\varepsilon[(\partial_{k}\varphi_2)_\varepsilon*S^t_\varepsilon(\widetilde{\mathfrak{B}}_{(d+1)kj}^* K_\varepsilon(\partial_j v_0))]^\varepsilon,$$
which has the same form as the former one, and $\mathfrak{B}_{d+1}^*$ is more regular than $\chi^*$. This means the latter term can be handled in the same way as the former one. Similarly, taking the former term as the test function into \eqref{parabolic_conver_es_Lw}, we find
\begin{align*}
  \nabla(\varepsilon [S_\varepsilon(\widetilde{\chi}^* K_\varepsilon(\nabla v_0))]^\varepsilon)=[(\nabla\varphi_2)_\varepsilon*S^t_\varepsilon(\widetilde{\chi}^* K_\varepsilon(\nabla v_0))]^\varepsilon+[S_\varepsilon(\nabla_y\widetilde{\chi}^* K_\varepsilon(\nabla v_0))]^\varepsilon,
\end{align*}
where these two terms are in the same form and $\widetilde{\chi}^*$ is more regular than $\nabla_y\widetilde{\chi}^*$. Thus, in view of $I_1-I_5$ in the proof of Lemma \ref{parabolic_conver_lem_Lw}, the key is to bound
\begin{align*}  Q_{31}&=\Big|\iint_{\Omega_T}[A-\widehat{A}]^\varepsilon(\nabla u_0-S_\varepsilon(\nabla u_0))\cdot[S_\varepsilon(\nabla_y\widetilde{\chi}^* K_\varepsilon(\nabla v_0))]^\varepsilon\Big|,\\ Q_{32}&=\iint_{\Omega_T}|[AS_\varepsilon(\nabla_y\widetilde{\chi} K_\varepsilon(\nabla u_0))-S_\varepsilon(A\nabla_y \widetilde{\chi} K_\varepsilon(\nabla u_0))]^\varepsilon||[S_\varepsilon(\nabla_y\widetilde{\chi}^* K_\varepsilon(\nabla v_0))]^\varepsilon|,\\ Q_{33}&=\varepsilon^2\iint_{\Omega_T}|\{S_\varepsilon[\partial_t\widetilde{\mathfrak{B}}_{d+1}K_\varepsilon(\nabla u_0)+\widetilde{\mathfrak{B}}_{d+1}\nabla S_\varepsilon(\partial_tu_0)\eta_\varepsilon]\}^\varepsilon||[S_\varepsilon(\nabla_y\widetilde{\chi}^* K_\varepsilon(\nabla v_0))]^\varepsilon|.
\end{align*}
The rest terms can be handled in the same manner. Especially, the terms of boundary layers can be estimated by \eqref{parabolic_conver_es_I2_boundary} and \eqref{parabolic_conver_es_v0_2}.

To bound $Q_{31}$, we write
\begin{align*}
  Q_{31}&\leq \Big|\iint_{\Omega_T}(\nabla u_0-S_\varepsilon(\nabla u_0))\cdot\{(A^*-\widehat{A}^*) S_\varepsilon[\nabla_y\widetilde{\chi}^* K_\varepsilon(\nabla v_0)]-S_\varepsilon[(A^*-\widehat{A}^*)\nabla_y\widetilde{\chi}^* K_\varepsilon(\nabla v_0)]\}^\varepsilon\Big|\\&\qquad+\Big|\iint_{\Omega_T}(\nabla u_0-S_\varepsilon(\nabla u_0))\cdot\{S_\varepsilon[(A^*-\widehat{A}^*)\cdot\nabla_y\widetilde{\chi}^*\cdot K_\varepsilon(\nabla v_0)]\}^\varepsilon\Big|,
\end{align*}
where the second term can be handled by Lemma \ref{parabolic_conver_lem_u0-Su0}. For the first term, one can see from estimate \eqref{parabolic_conver_smoothing_coro_es_1} that it can be bounded by
\begin{align*}
  &\quad C\varepsilon\{\|\nabla u_0\|_{L^{2, p_0}(\Omega_T)}+\|\nabla u_0\|_{L^{4, 2}(\Omega_T)}\}\\&\qquad\qquad\cdot\{\|\nabla_y\widetilde{\chi}^* K_\varepsilon(\nabla v_0)\|_{L^{2, q_0^*}(\Omega_T; \bm{L^{\bar{q}}})}+\|\nabla_y\widetilde{\chi}^* K_\varepsilon(\nabla v_0)\|_{L^{4, 2}(\Omega_T; \bm{L^{\bar{q}}})}\}\\&\leq C\varepsilon\{\|\nabla u_0\|_{L^2(0, T; \dot{W}^{1, q_0}(\Omega))}+\|\partial_t u_0\|_{L^2(0, T; L^{q_0}(\Omega))}\}\cdot\{\|\nabla v_0\|_{L^2(0, T; \dot{W}^{1, q_0}(\Omega))}+\|\partial_tv_0\|_{L^2(0, T; L^{q_0}(\Omega))}\},
\end{align*}
where Lemma \ref{parabolic_conver_boundary_lem_embedding} was also used. Therefore, it follows from estimate \eqref{parabolic_conver_es_v0_1} that
\begin{align*}
  Q_{31}\leq C\varepsilon\{\|\nabla u_0\|_{L^2(0, T; \dot{W}^{1, q_0}(\Omega))}+\|\partial_t u_0\|_{L^2(0, T; L^{q_0}(\Omega))}\}\cdot\|F\|_{L^2(0, T; L^{q_0}(\Omega))}.
\end{align*}
For $Q_{32}$, by using \eqref{parabolic_conver_smoothing_coro_es_4}, we see that
\begin{align*}
  Q_{32}&\leq C\varepsilon\{\|\nabla_y\widetilde{\chi}K_\varepsilon(\nabla u_0)\|_{L^{2, q_0^*}(\Omega_T; \bm{L^{\bar{q}}})}\|\nabla_y\widetilde{\chi}^*K_\varepsilon(\nabla v_0)\|_{L^{2, p_0}(\Omega_T; \bm{L^{\bar{q}}})}\\&\qquad+\|\nabla_y\widetilde{\chi}K_\varepsilon(\nabla u_0)\|_{L^{4, 2}(\Omega_T; \bm{L^{\bar{q}}})}\|\nabla_y\widetilde{\chi}^*K_\varepsilon(\nabla v_0)\|_{L^{4, 2}(\Omega_T; \bm{L^{\bar{q}}})}\}\\&\leq C\varepsilon\{\|\nabla u_0\|_{L^2(0, T; \dot{W}^{1, q_0}(\Omega))}+\|\partial_t u_0\|_{L^2(0, T; L^{q_0}(\Omega))}\}\cdot\|F\|_{L^2(0, T; L^{q_0}(\Omega))},
\end{align*}
where Lemmas \ref{parabolic_pre_lem_chi}, \ref{parabolic_conver_boundary_lem_embedding} and \ref{parabolic_conver_lem_v0} were used and $C$ depends only on $d, \mu, \Omega$, $[A]_{\mathscr{W}^{\sigma, d}(\Omega_T; \bm{L^\infty})}$, the $\mathrm{VMO}_x$ character of $\widehat{A}$.
 To deal with $Q_{33}$, by Lemma \ref{parabolic_pre_smooth_lem_2}, we have
\begin{align*}
  Q_{33}&\leq C\varepsilon^2\|\partial_t\widetilde{\mathfrak{B}}_{d+1}K_\varepsilon(\nabla u_0)\|_{L^{4/3, 2}(\Omega_T; \bm{L^2})}\|\nabla v_0\|_{L^{4, 2}(\Omega_T)}\\&\qquad+C\varepsilon^2\|\widetilde{\mathfrak{B}}_{d+1}\nabla S_\varepsilon(\partial_tu_0)\eta_\varepsilon\|_{L^{2, q_0}(\Omega_T; \bm{L^2})}\|\nabla v_0\|_{L^{2, p_0}(\Omega_T)}\\&\leq C\varepsilon\{\|\nabla u_0\|_{L^2(0, T; \dot{W}^{1, q_0}(\Omega))}+\|\partial_t u_0\|_{L^2(0, T; L^{q_0}(\Omega))}\}\cdot\|F\|_{L^2(0, T; L^{q_0}(\Omega))},
\end{align*}
where we have also used Lemmas \ref{parabolic_pre_lem_fkB}, \ref{parabolic_pre_smooth_lem_partt}, \ref{parabolic_conver_boundary_lem_embedding} and \ref{parabolic_conver_lem_v0}. Consequently, we have
\begin{align*}
  Q_3\leq C\varepsilon\{\|\nabla u_0\|_{L^2(0, T; \dot{W}^{1, q_0}(\Omega))}+\|\partial_t u_0\|_{L^2(0, T; L^{q_0}(\Omega))}\}\cdot\|F\|_{L^2(0, T; L^{q_0}(\Omega))}.
\end{align*}

In view of the estimates of $Q_1, Q_2, Q_3$ and \eqref{parabolic_conver_es_I}, we get
\begin{align*}
  \Big|\iint_{\Omega_T} w_\varepsilon\cdot F dxdt\Big|\leq C\varepsilon\{\|\nabla u_0\|_{L^2(0, T; \dot{W}^{1, q_0}(\Omega))}+\|\partial_t u_0\|_{L^2(0, T; L^{q_0}(\Omega))}\}\cdot\|F\|_{L^2(0, T; L^{q_0}(\Omega))},
\end{align*}
which, by duality, yields the desired result. The proof is completed.
\end{proof}

\section*{Acknowledgements}
The author is much obliged to Professor Zhongwei Shen for the guidance.

%% References
\bibliographystyle{model1-num-names}%"model1-num-names" also works
\bibliography{bib}

\end{document}